\documentclass[11pt,twoside]{article}

\usepackage{mathtools}

\usepackage{amsmath,amssymb,amsthm}
\usepackage{enumerate}
\usepackage{graphicx}

\numberwithin{equation}{section}
\newtheorem{theorem}{Theorem}[section]
\newtheorem{assumption}{Assumption}

\newtheorem{lemma}[theorem]{Lemma}
\newtheorem{proposition}[theorem]{Proposition}

\theoremstyle{remark}
\newtheorem{remark}[theorem]{Remark}

\newcommand{\bke}[1]{\left ( #1 \right )}
\DeclarePairedDelimiter{\norm}{\lVert}{\rVert}
\newcommand{\R}{\mathbb{R}}
\renewcommand{\Re} {\mathop{\mathrm{Re}}}

\newcommand{\ntt}{N([t,\infty))}
\newcommand{\stt}{S([t,\infty))}
\newcommand{\sourceterm}{H}
\newcommand{\profile}{W}

\title{Fast-moving finite and infinite trains of solitons\\ for
nonlinear Schr\"odinger equations}
\author{
Stefan Le Coz,%
\thanks{Institut de Math\'ematiques de Toulouse, Universit\'e Paul Sabatier, 118 route de Narbonne, 31062 Toulouse Cedex 9, France, E-mail address: \texttt{slecoz@math.univ-toulouse.fr}}
\and 
Dong Li,%
\thanks{
Department of Mathematics, University of British Columbia, Vancouver BC
Canada V6T 1Z2, and Institute for Advanced Study, 1st Einstein Drive, Princeton, NJ
08544, USA, E-mail address: \texttt{mpdongli@gmail.com}}
\and Tai-Peng Tsai%
\thanks{
Department of Mathematics, University of British Columbia, Vancouver BC
Canada V6T 1Z2, E-mail address: \texttt{ttsai@math.ubc.ca}}
}

\date{}
\begin{document}

%


\maketitle

\begin{abstract}
%
We study \emph{infinite soliton trains} solutions of nonlinear Schr\"odinger equations (NLS), i.e. solutions behaving at large time as the sum of infinitely many solitary waves. 
Assuming
the composing solitons have sufficiently large relative speeds, we prove the existence and uniqueness of such a soliton train. 
We also give a new construction of multi-solitons (i.e. finite trains)
and prove uniqueness in an exponentially small neighborhood, and  we consider the case of solutions composed of several solitons and kinks (i.e. solutions with a non-zero background at infinity).

\smallskip
\noindent{\it Keywords:}\quad soliton train, multi-soliton, multi-kink,
nonlinear Schr\"odinger equations.

\smallskip
\noindent{\it 2010 Mathematics Subject Classification:} 35Q55(35C08,35Q51).
\end{abstract}

\section{Introduction}
We consider the following nonlinear Schr\"odinger equation (NLS):
\begin{align}
i\partial_t u +\Delta u = -g (|u|^2) u =:-f(u), \label{v100a0}
\end{align}
where $u=u(t,x)$ is a complex-valued function on $\mathbb R \times
\mathbb R^d$, $d\ge 1$.

The purpose of this paper is to construct special families of
solutions to the energy-subcritical NLS
\eqref{v100a0}. We will look for \emph{infinite soliton trains, multi-solitons} and \emph{multi-kinks} solutions.  

Recall that it is generically expected that global solutions to nonlinear dispersive equations like NLS eventually decompose at large time as a sum of solitons plus a scattering remainder (\emph{Soliton Resolution Conjecture}). Except for the specific case of integrable equations, such results are usually out of reach (see nevertheless the recent breakthrough on energy-critical wave equation \cite{DuKeMe12}). In the case of nonlinear Schr\"odinger equations, multi-solitons can be constructed via the inverse scattering transform in the integrable case ($d=1$, $f(u)=|u|^2u$). In non-integrable frameworks, multi-solitons are known to exist since the pioneering work of Merle \cite{Me90} (see Section \ref{subsec:multi} for more details on the existing results of multi-solitons). The multi-solitons constructed up to now were made of a finite number of solitons and there was little evidence of the possibility of existence of infinite trains of solitons (note nevertheless the result \cite{Ka95} in the integrable case). The existence of such infinite solitons trains is however important, as they may provide examples or counter-examples of solutions with borderline behaviors (as it is the case for the Korteweg-de Vries equation, see \cite{MaMe05}). In this paper, we show the existence of such infinite soliton trains for power nonlinearities.
It turns out that our strategy  is very flexible and allows us to prove many results of existence and uniqueness of multi-solitons and multi-kinks solutions for generic nonlinearities. In the rest of this introduction, we state our main results on infinite trains (Section \ref{subsec:infinite}), multi-solitons (Section \ref{subsec:multi}) and multi-kinks (Section \ref{subsec:kinks}) and give a summary of the strategy of the proofs (Section \ref{subsec:strategy}).

\subsection{Infinite soliton trains}\label{subsec:infinite}

Our first main result is on the construction of a solution to \eqref{v100a0} behaving at large time like a sum of  infinitly many solitons. For this purpose we have to
use scale invariance and work with the power nonlinearity
$f_1(u)=|u|^{\alpha}u$ , $0<\alpha<\alpha_{\max}$, $\alpha_{\max} = +\infty$ for $d=1,2$
and $\alpha_{\max}= \frac 4 {d-2}$ for $d\ge 3$. Let $\Phi_0 \in
H^1(\mathbb R^d)$ be a fixed bound state which solves the elliptic
equation
\begin{align*}
-\Delta \Phi_0 + \Phi_0 -|\Phi_0|^{\alpha} \Phi_0=0.
\end{align*}
For $j\ge 1$, $\omega_j>0$ (\emph{frequency}), $\gamma_j \in \mathbb R$ (\emph{phase}), $v_j \in
\mathbb R^d$ (\emph{velocity}), define a soliton $\tilde R_j$ by
\begin{equation}\label{soliton_100}
\tilde R_j(t,x):=e^{i(\omega_j
t-\frac{|v_j|^2}{4}+\frac{1}{2}v_j\cdot
x+\gamma_j)}\omega_j^{\frac{1}{\alpha}} \Phi_0 \big(\sqrt{\omega_j}
(x-v_jt)\big).
\end{equation}
We consider the following soliton train:
\begin{align}
R_\infty = \sum_{j=1}^{\infty} \tilde R_j. \label{train_001}
\end{align}
Since \eqref{v100a0}
is a nonlinear problem, the function $ R_\infty=R_\infty(t,x)$ is no longer a
solution in general. Nevertheless we are going to show that in the vicinity of $R_\infty$ one can
still find a solution $u$  to \eqref{v100a0} to which we refer to as an \emph{infinite soliton train}. More precisely, the solution $u$ to
\eqref{v100a0} is defined on $[T_0,+\infty)$ for some $T_0 \in \R$ and
such that
\begin{align}
\lim_{t\to +\infty} \| u -R_\infty \|_{X([t,\infty)\times \R^d)}=0.
\label{v100a7}
\end{align}
Here $\|\cdot \|_{X([t,\infty)\times \R^d)}$ is some space-time norm
measured on the slab $[t,\infty)\times \R^d$. A simple example is
$X= L^\infty_t L^2_x$ in which case one can replace \eqref{v100a7} by the
equivalent condition
\begin{align}
\lim_{t\to +\infty} \| u(t) -R(t) \|_{L^2}=0. \notag
\end{align}
However the definition \eqref{v100a7} is more flexible as one can
allow general Strichartz spaces (see \eqref{strichartz_def}).

The main idea is that in the energy-subcritical setting, all
solitons have exponential tails (see \eqref{v100a5}). When their
relative speed is large, these traveling solitons are well-separated
and have very small overlaps which decay exponentially in time. At
such high velocity and exponential separation, one does not need
fine spectral details and the whole argument can be carried out as a
perturbation around the desired profile (e.g. the soliton series $R$) in a well-chosen function
space. As our proof is based on contraction estimates, the
uniqueness follows immediately, albeit in a very restrictive function
class.

We require that the
parameters $( \omega_j,v_j)$ of the train satisfy the following assumption.

\begin{assumption}\label{assumption}$\phantom{1}$
\begin{itemize}
\item (Integrability) There exists $r_1\ge 1$, $\frac{d\alpha}2 < r_1 < \alpha+2$, such that
\begin{equation}
 A_{\omega}:= \sum_{j=1}^\infty 
 \omega_j^{\frac 1{\alpha} -\frac d {2r_1}} <\infty. \label{omega_50}
 \end{equation}
\item (High relative speeds)  The solitons travel sufficiently fast: there exists a constant $v_\star >0$ such that
\begin{align}
\sqrt { \min \{ \omega_j,\omega_k \} } \Bigl( |v_k -v_j |   \Bigr)
\ge v_\star   , \qquad \forall\, j\ne k. \label{omega_star_100}
\end{align}
\end{itemize}
\end{assumption}

Since $R_\infty$ may be badly localized, we seek a infinite soliton train
solution to \eqref{v100a0} in the form $u= R_\infty+\eta$, where $\eta$
satisfies the perturbation equation
\begin{align*}
i\partial_t \eta  + \Delta \eta  = -f( R_\infty+\eta ) + \sum_{j=1}^{\infty}
f(\tilde R_j).
\end{align*}
In Duhamel formulation, the
perturbation equation for $\eta $ reads
\begin{align}
\eta (t)= i \int_t^{\infty} e^{i(t-\tau)\Delta}\Bigl( f(R_\infty+\eta )
-\sum_{j=1}^{\infty} f(\tilde R_j) \Bigr) d\tau, \qquad\forall\,
t\ge 0. \label{g_main_001}
\end{align}

 The following theorem gives
the existence and uniqueness of the solution $\eta $ to
\eqref{g_main_001}.

\begin{theorem}[Existence of an infinite soliton train solution] 
\label{thm_inf_1}
Consider \eqref{v100a0} with $f(u)=|u|^{\alpha} u$ satisfying
$0<\alpha<\alpha_{\max}$.  Let $R_\infty$ be given as in
\eqref{train_001}, 
with parameters $\omega_j>0$,
$\gamma_j\in \R$, and $v_j\in \R^d$ for $j \in \mathbb{N}$, which satisfy
Assumption \ref{assumption}. There exist constants $C>0$, $c_1>0$ and
$v_{\sharp}\gg 1$
 such that $($see \eqref{omega_star_100}$)$ if $v_\star
>v_{\sharp}$, then there exists  a unique
solution $\eta \in S([0,\infty))$ 
(see \eqref{strichartz_def} for the  definition of Strichartz space) 
to \eqref{g_main_001} satisfying
\begin{equation}\label{eq:for-uniqueness}
\|\eta \|_{S([t,\infty))} +\|\eta (t)\|_{L^{\alpha+2}} \le C e^{-c_1 v_\star  t},
\qquad \forall\, t\ge 0.
\end{equation}
\end{theorem}

\begin{remark}
By using Theorem \ref{thm_inf_1} and Lemma \ref{lem_train_001}, one
can justify the existence of a solution $u= R_\infty+\eta $ satisfying
\eqref{v100a0} in the distributional sense. The uniqueness of such
solutions is only proven for the perturbation $\eta $ satisfying
\eqref{g_main_001} and \eqref{eq:for-uniqueness}. In the mass-subcritical case $0<\alpha<\frac 4d$,
the soliton train $ R_\infty$ is in the Lebesgue space $C_{t}^0 L_x^2
\cap L_{tx}^{\infty}$, and one can show that the solution $u=
R_\infty+\eta $ can be extended to all $\mathbb R\times \mathbb R^d$
and satisfies $u\in C_t^0 L_x^2(\mathbb R\times \mathbb R^d) \cap
L_{t, loc }^{\frac{2(d+2)}d} L_x^{\frac{2(d+2)}d} (\mathbb R\times
\mathbb R^d)$ (see \eqref{def_lqloc}). Hence it is a localized
solution in the usual sense. In the mass-supercritical case $\frac
4d\le \alpha<\alpha_{\max}$, the soliton train
$R_\infty=\sum_{j=1}^\infty\tilde R_j$ is no longer in $L^2$ since
each composing piece $\tilde R_j$ has $O(1)$ $L^2$-norm. Nevertheless
we shall still build a regular solution to \eqref{g_main_001} since $
R_\infty$ has Lebesgue regularity $L_t^\infty L_x^{\frac{d\alpha}2+}
\cap L_{tx}^\infty$ which is enough for the perturbation argument to
work. We stress that in this case the solution $\eta $ is only defined
on $[0,\infty)\times \mathbb R^d$ and scatters forward in time in
  $L^2$.
\end{remark}

\begin{remark}
Typically the parameters $(\omega_j,v_j)$ are chosen in the
following order: first we take $(\omega_j)$ satisfying \eqref{omega_50};
then we inductively choose $v_j$ such that the condition
\eqref{omega_star_100} is satisfied. For example one can take for
$j\ge 1$,
 $\omega_j=2^{-j}$ and $v_j= 2^j\bar v $ for $\bar v\in\R^d$, $|\bar v|=v_\star $.
Note that, when $0<\alpha<\frac 4d$ (mass-subcritical case), we can
choose $r_1 \le 2$.  The soliton train is then in $L^\infty_t L^2_x$.
We require $\frac{d\alpha}2 < r_1 $ so that the exponent in
\eqref{omega_50} is positive. The condition $ r_1 < \alpha+2$ will be
needed to show \eqref{19_002} in Lemma \ref{lem_train_001}.
\end{remark}

\begin{remark}
 Note that
we did not introduce initial positions in the definition of $\tilde R_j$, so each soliton starts centered at $0$. With some minor
modifications, our construction can also work for the general case with the solitons starting centered at various $x_j$. For simplicity of presentation we shall not state the
general case here. 
\end{remark}

\begin{remark}
Certainly Theorem \ref{thm_inf_1} can hold in more general
situations. For example instead of taking a fixed profile $\Phi_0$ in
\eqref{soliton_100}, one can draw $\Phi_0$ from a finite set of
profiles $\mathcal A=\{ \Phi_0^1,\cdots, \Phi_0^K \}$ where each
$\Phi_0^j$ is a bound state.  
\end{remark}

\begin{remark}
The rate of spatial decay of multi-solitons is still an open question in the NLS case  (for KdV it is partly known: multi-solitons decay exponentially on the right). In Theorem \ref{thm_inf_1}, the soliton train profile $R_\infty$ around which we build our solution has only a polynomial spatial decay, not uniform in time. Hence we expect the solution $u=R_\infty+\eta$ to have the same decay.
\end{remark}

\subsection{Multi-solitons}\label{subsec:multi}

From now on, we work with a generic nonlinearity and just assume that  $f(u)=g(|u|^2)u$ where the function $g:\; [0,\infty) \to \R$ obeys some H\"older
conditions mimicking the usual power type nonlinearity. Precisely,
\begin{itemize}
\item $g \in C^0([0,\infty), \R) \cap C^2((0,\infty),\R)$, $g(0)=0$ and
\begin{align} \label{v100a1}
 |sg^{\prime}(s)| +|s^2 g^{\prime\prime}(s)| \le C \cdot (s^{\alpha_1} +s^{\alpha_2}), \qquad\forall\, s>0,
\end{align}
where $C>0$, $0<\alpha_1 \le \alpha_2< \frac{\alpha_{\max}}2$.
\end{itemize}

A typical example is $g(s)=s^{\alpha}$ for some
$0<\alpha<\frac{\alpha_{\max}}{2}$.
A useful example to keep in mind is the combined
nonlinearity $g(s)=s^{{\alpha_1}} - s^{{\alpha_2}}$ for
some $0<\alpha_1<\alpha_2<\frac{\alpha_{\max}}{2}$. Other examples can be
easily constructed. 
Throughout the rest of this paper we shall
assume $f(u)=g(|u|^2)u$ satisfy \eqref{v100a1}.
 In this case the corresponding
nonlinearity $f(u)$ is usually called energy-subcritical since there are
lower bounds of the
lifespans of the $H^1$ local solutions which depend only on the $H^1$-norm
(not the profile) of initial data (cf. \cite{Ca03,GiVe79}). The condition
\eqref{v100a1} is a natural generalization of the pure
power nonlinearities. For much of our analysis it can be replaced by the
weaker condition that $g(s)$ and $sg^{\prime}(s)$ are  H\"older continuous
with suitable exponents. However \eqref{v100a1} is fairly easy to check and it suffices
for most applications.

We give a definition of a solitary wave slightly more general than \eqref{soliton_100}. Given a set of
parameters $\omega_0>0$ (\emph{frequency}), $\gamma_0 \in \R$ (\emph{phase}),
$x_0$, $v_0\in \R^d$ (\emph{position} and \emph{velocity}), a \emph{solitary wave}, or a
\emph{soliton}, is a solution to \eqref{v100a0} of the form
\begin{align}
R_{\Phi_0, \omega_0, \gamma_0, x_0,v_0}:= \Phi_0(x-v_0t-x_0 ) \exp
\left( i \Bigl( \frac 12 v_0 \cdot x - \frac 14 |v_0|^2 t + \omega_0
t +\gamma_0 \Bigr) \right), \label{v100a3a}
\end{align}
where $\Phi_0 \in H^1(\mathbb R^d)$ solves the elliptic equation
\begin{align}
-\Delta \Phi_0 + \omega_0 \Phi_0 -f(\Phi_0) =0. \label{v100a4}
\end{align}
A nontrivial $H^1$ solution to \eqref{v100a4} is usually called a
\emph{bound state}. Compared with \eqref{soliton_100}, the main difference is that we do not use the parameter $\omega_j$ to rescale the solitons.

Existence of bound states is guaranteed (see \cite{BeLi83-1}) if we assume, in addition to  \eqref{v100a0},  that
there exists $s_0>0$, such that
\begin{equation}
G(s_0) := \int_0^{s_0} g(\tilde s) d\tilde s > \omega_0s_0. \label{v100a3}
\end{equation}
Note that the  condition \eqref{v100a3}
makes the nonlinearity focusing. 

All bound states are exponentially decaying (cf.
Section 3.3 of \cite{BeSh91} for example), i.e.
\begin{align}
e^{\sqrt{\omega}|x|} ( |\Phi_0| + |\nabla \Phi_0| ) \in
L^{\infty}(\R^d), \quad \text{for all $0<\omega<\omega_0$.} \label{v100a5}
\end{align}
A \emph{ground state} is a bound state which minimizes among all bound states  the \emph{action}
\begin{align}
S(\Phi_0)= \frac 12 \| \nabla \Phi_0\|_2^2 + \frac{\omega_0}2 \|
\Phi_0 \|_2^2 - \frac  12 \int_{\mathbb R^d} G(|\Phi_0|^2) dx. \notag
\end{align}
The ground state is usually unique modulo symmetries of the
equation (see e.g. \cite{Mc93} for precise conditions on the nonlinearity ensuring uniqueness of the ground state). If $d\geq2$ there exist infinitely many other solutions called \emph{excited states}  (see \cite{BeLi83-1,BeLi83-2} for more on ground states and excited states). The
corresponding solitons are usually termed \emph{ground state solitons}
(resp. \emph{excited state solitons}).

A \emph{multi-soliton} is a solution to
\eqref{v100a0} which roughly speaking looks like the sum of $N$
solitons. To fix notations, let (see \eqref{v100a3a})
\begin{align}
R(t,x)= \sum_{j=1}^N R_{\Phi_j,\omega_j,\gamma_j,x_j,v_j}(t,x) =:
\sum_{j=1}^N R_j(t,x), \label{v100a6}
\end{align}
where each $R_j$ is a soliton made from some parameters $(\omega_j,\gamma_j, x_j, v_j)$ and bound state $\Phi_j$ (we assume that \eqref{v100a3} holds true for all $\omega_j$).

If
each $\Phi_j$ in \eqref{v100a6} is a ground state, then the
corresponding multi-soliton is called a \emph{ground state multi-soliton}.
If at least one $\Phi_j$ is an excited state, we call it 
an \emph{excited state multi-soliton}.

We now review in more details some known results on multi-solitons. Most results are
on the pure power nonlinearity $f(u)=|u|^{\alpha}u$ with
$0<\alpha<\alpha_{\max}$ and ground states. If $\alpha=\frac 4d$ (resp.
$\alpha<\frac 4d$, $\alpha>\frac 4d$), then equation
\eqref{v100a0} is called ($L^2$) mass-critical (resp.
mass-subcritical, mass-supercritical). In the integrable case $d=1$,
$\alpha=2$, Zakharov and Shabat \cite{ZaSh72} derived an explicit
expression of multi-solitons by using  the inverse scattering
transform. For the mass-critical NLS, which is non-integrable in higher dimensions, Merle
\cite{Me90} (see Corollary 3 therein) constructed a solution blowing up at exactly $N$ points at the same time, which gives a multi-soliton
after a pseudo-conformal transformation. In
the mass-subcritical case, the ground state solitary waves are
stable. Assuming the composing solitary waves $R_j$ are ground
states and have different velocities (i.e. $v_j \ne v_k$ if $j\ne k$
in \eqref{v100a6}), Martel and Merle \cite{MaMe06} proved the
existence of an $H^1$ ground state multi-soliton $u\in C([T_0,\infty),
H^1)$ such that
\begin{align}
\norm[\Big]{ u(t) -\sum_{j=1}^N R_j(t) }_{H^1} \le C
e^{-\beta\sqrt{\omega_\star }v_\star t}, \qquad \forall\, t\ge
T_0, \label{v100a9}
\end{align}
for some constant $\beta>0$,
where $T_0\in\R$ is large enough, 
and the minimal relative velocity $v_\star $ and the minimal frequency $\omega_\star $ are defined by
\begin{align}
v_\star&:= \min\{ |v_j-v_k|:\;\; 1\le j\ne k\le N \} ,\label{v100a11}\\
\omega_\star &=\min\{\omega_j, \; 1\le j\le N\}.
\end{align}
In the same work, the
authors also considered a general energy-subcritical nonlinearity
$f(u)=g(|u|^2)u$ with $g\in C^1$, $g(0)=0$ and satisfy $\| s^{-\alpha} g^{\prime}(s) \|_{L_s^{\infty}(s\ge 1)} <\infty$
for some $0<\alpha<\alpha_{\max} /2$.
Assuming a nonlinear stability condition around the
ground state (see (16) of \cite{MaMe06}), they proved the existence of
an $H^1$ ground state multi-soliton satisfying the same estimate
\eqref{v100a9}. 

In \cite{CoMaMe11}, C\^{o}te, Martel and Merle considered the
 mass-supercritical NLS ($f(u)=|u|^{\alpha}u$ with $\frac
4d<\alpha<\alpha_{\max}$). Assuming the ground state solitons $R_j$
have different velocities, the authors constructed an $H^1$ ground
state multi-soliton $u$ satisfying \eqref{v100a9}.
This result was sharpened in 1D by Combet: in \cite{Co10}, he showed the existence of a $N$-parameters family of multi-solitons.

In
\cite{CoLC11}, C\^{o}te and Le Coz considered the general
energy-subcritical NLS with $f(u)=g(|u|^2)u$ satisfying 
assumptions similar to \eqref{v100a1} and \eqref{v100a3}. Assuming the solitary
waves $R_j$ are excited states and have large relative velocities,
i.e. assuming
\[
v_\star \geq v_\sharp>0
\]
for $v_\sharp$ large enough, the authors constructed an excited
state multi-soliton $u\in C([T_0,\infty), H^1)$ for $T_0\in \R$ large enough,
which also satisfies  \eqref{v100a9}.

The main strategy used in the above mentioned works
\cite{CoLC11,CoMaMe11,MaMe06,Me90} is the following: one takes a
sequence of approximate solutions $u_n$ solving \eqref{v100a0} with
final data $u_n(T_n)=R(T_n)$, $T_n\to \infty$; by using local
conservation laws and coercivity of the Hessian (this has to be
suitably modified in certain cases, cf. \cite{CoLC11}), one derives
uniform $H^1$ decay estimates of $u_n$ on the time interval $[T_0,
T_n]$ where $T_0$ is independent of $n$; the multi-soliton is then
obtained after a compactness argument. We should point 
out that
the uniqueness of multi-solitons is still left open by the above
analysis (see nevertheless \cite{Co10,CoLC11} for existence of a $1$ and $N$ parameters families of multi-solitons). Under restrictive assumptions on the nonlinearity (e.g. high regularity or flatness assumption at the origin) and a large relative speeds hypothesis, stability of multi-solitons was obtained in \cite{MaMeTs06, Pe97,Pe04,RoScSo03} and instability in \cite{CoLC11}. 
See also
Remark \ref{rem_instable} below.

In this section we give new constructions of multi-solitons.  We work in the context of the
energy-subcritical problem \eqref{v100a0} with $f(u)$ satisfying
\eqref{v100a1} and \eqref{v100a3}
We shall focus on \emph{fast-moving}
solitons, i.e. the minimum relative velocity $v_\star $ defined in
\eqref{v100a11} is sufficiently large. The composing solitons are in
general bound states which can be either ground states or excited
states. In
our next two results, we recover and improve the result from  
\cite[Theorem 1]{CoLC11} in various
settings.  The improvements here are the lifespan and uniqueness.  As for the infinite train, our new proof rely on a contraction argument around the desired profile. We begin
with the pure power nonlinearity case.

\begin{theorem}[Existence and uniqueness of multi-solitons, power nonlinearity case]
\label{thm_N_1} Consider \eqref{v100a0} with $f(u)=|u|^{\alpha} u$
satisfying $0<\alpha<\alpha_{\max}$. Let $R$ be the same as in
\eqref{v100a6} and define $v_\star $ as in \eqref{v100a11}. There exists
constants $C>0$, $c_1>0$ and $v_{\sharp}\gg 1$  such that if
$v_\star >v_{\sharp}$, then there exists a unique solution $u\in
C([0,\infty), H^1)$ to \eqref{v100a0} satisfying
\begin{align*}
e^{c_1 v_\star t} \| u-R \|_{S([t,\infty))} +e^{c_2 v_\star t} \|
  \nabla (u -R) \|_{S([t,\infty))} \le C , \qquad \forall\, t\ge 0.
\end{align*}
Here $c_2=c_1 \cdot \min(1,\alpha) \le c_1$.
In particular $\| u(t)-R(t) \|_{H^1} \le C e^{-c_2 v_\star t}$.
\end{theorem}

\begin{remark}
As was already mentioned, Theorem \ref{thm_N_1} is a slight
improvement of a corresponding result (Theorem 1) in \cite {CoLC11}.
Here the multi-soliton is constructed on the time interval
$[0,\infty)$ whereas in \cite{CoLC11} this was done on $[T_0,\infty)$
for some $T_0>0$ large. In particular, we do not have to wait for the interactions between the solitons to be small to have existence of our multi-soliton. However, we have no control on the constant $C$ so at small times our multi-soliton may be very far away from the sum of solitons. The uniqueness of solutions is a subtle
issue, see Remark \ref{rem_instable}.
\end{remark}

The next result concerns the general nonlinearity $f(u)$.

\begin{theorem}[Existence and uniqueness of multi-solitons, general 
nonlinearity case]
\label{thm_N_2} Consider \eqref{v100a0} with $f(u)=g(|u|^2)u$
satisfying \eqref{v100a1} and \eqref{v100a3}. Let $R$ be the same as
in \eqref{v100a6} and define $v_\star $ as in \eqref{v100a11}. There
exist constants $C>0$, $c_1>0$, $c_2>0$, $T_0\gg1$ and $v_{\sharp}\gg 1$, such
that if $v_\star >v_{\sharp}$, then there is a unique solution $u\in
C([T_0,\infty), H^1)$ to \eqref{v100a0} satisfying
\begin{align*}
e^{c_1 v_\star t} \| u-R \|_{S([t,\infty))} +e^{c_2 v_\star t} \|
  \nabla (u -R) \|_{S([t,\infty))} \le C , \qquad \forall\, t\ge T_0.
\end{align*}
\end{theorem}

\begin{remark}\label{rem_T0}
Unlike Theorem \ref{thm_N_1}, the solution in Theorem \ref{thm_N_2}
exists only for $t\ge T_0$ with $T_0$ sufficiently large.  To take
$T_0=0$, our method requires extra conditions. For such results see
Section \ref{sec:T0}. We can also extend Theorem \ref{thm:multi-kinks} similarly.
\end{remark}

\begin{remark} \label{rem_instable}
In Theorems \ref{thm_N_1} and \ref{thm_N_2}, the uniqueness of the multi-soliton
solution holds in a quite restrictive function class whose Strichartz-norm decay as $e^{-c_1  v_\star t}$.
 A natural question is
whether uniqueness holds in a wider setting. In general this is a
very subtle issue and in some cases one cannot get away with the
exponential decay condition. In \cite{CoLC11}, the authors considered the case
when one of the composing soliton, say $R_1$ is unstable. Assuming
$g\in C^{\infty}$ (see \eqref{v100a0}) and the operator $L=-i\Delta
+ i\omega_1 -i df(\Phi_1)$ has an eigenvalue $\lambda_1 \in \mathbb C$
with $\rho:=\Re(\lambda_1)>0$, they constructed a one-parameter family
of multi-solitons $u_{a}(t)$ such that for some $T_0=T_0(a)>0$,
\begin{align*}
\| u_a(t) - \sum_{j=1}^N R_j(t) -a Y(t) \|_{H^1(\mathbb R^d)} \le C
e^{-2\rho t}, \qquad \forall\, t\ge T_0.
\end{align*}
Here $Y(t)$ is a nontrivial solution of the linearized flow around
$R_1$, and $e^{\rho t} \| Y(t)\|_{H^1}$ is periodic in $t$. This
instability result shows that the exponential decay condition in
the uniqueness statement cannot be removed in general for 
NLS with unstable solitary waves.
\end{remark}

\subsection{Multi-kinks}\label{subsec:kinks}

In this subsection, we push our approach further and attack the problem of the existence of multi-kinks, i.e. solutions built upon solitons and  their nonlocalized counterparts the kinks. 
Before stating our result, let us first mention some related works. When its solutions are considered with a non-zero background  (i.e. $|u|\to \nu\neq0$  at
$\pm\infty$  ), the NLS equation \eqref{v100a0} is often refered to as the Gross-Pitaevskii equation. For general non-linearities, Chiron \cite{Ch12}  investigated the existence of traveling wave solutions with a non-zero background and showed that various types of nonlinearities can lead to a full zoology of profiles for the traveling waves. In the case of the ``classical'' Gross-Pitaevskii equation, i.e. 
 when $f(u)=(1-|u|^2)u$ and solutions verify $|u|\to 1$ at
infinity, the profiles of the traveling kink solutions  $K(t,x)=\phi_c(x-ct)$ are explicitly known and given for $|c|<\sqrt{2}$ by the formula
\[
\phi_c(x)=\sqrt{\frac{2-c^2}{2}}\operatorname{tanh}\left(\frac{x\sqrt{2-c^2}}{2}\right)+i\frac{c}{\sqrt{2}}
\]
with $\omega=0$.
(in particular, one can see that the limits at $-\infty$ and $+\infty$ are different, thus justifying the name ``kink''). 
In \cite{BeGrSm12}, B\'ethuel, Gravejat and Smets proved the stability forward in time of a profile composed of several kinks traveling at different speeds. Note that, due to the non-zero background of the kinks, the profile cannot be simply taken as a sum of kinks and one has to rely on another formulation of the Gross-Pitaevskii equation to define properly what is a multi-kink.

The main differences between our analysis and the works above mentioned are, first, that our kinks have a zero  background on one side and a non-zero one on the other side, and second, that, due to the Galilean transform used to give a speed to the kink, our kinks have infinite energy (due to the non-zero background, the rotation in phase generated by the Galilean transform is not killed any more by the decay of the modulus). In particular, this would prevent us to use energy methods as it was the case for multi-solitons in \cite{CoLC11,CoMaMe11,MaMe06} or multi-kinks \cite{BeGrSm12}.

 We place ourselves in dimension $d=1$. In such context and under suitable assumptions on the nonlinearity $f$, \eqref{v100a0} admits kink solutions. More precisely,
given $\gamma,\omega,v,x_0\in\R$, 
what we call a \emph{kink} solution of \eqref{v100a0}  (or \emph{half-kink}) is a function $K=K(t,x)$ defined similarly as a soliton by
\begin{equation*}
K(t,x):=e^{i\bigl( \frac 12 v x - \frac 14 |v|^2 t + \omega
t +\gamma \bigr)}\phi(x-vt-x_0),
\end{equation*}
but where $\phi$ satisfies the profile equation on $\R$ \emph{with a non-zero boundary condition} at one side of the real line, denoted by $\pm\infty$ and zero boundary condition on the other side (denoted by $\mp\infty$):
\begin{equation}\label{eq:snls}
\left\{
\begin{aligned}
&-\phi''+\omega\phi-f(\phi)=0,\\
&\lim_{x\to\mp\infty}\phi(x)=0,\quad \lim_{x\to\pm\infty}\phi(x)\neq0.
\end{aligned}
\right.
\end{equation}
The existence of half-kinks is granted by the following proposition.

\begin{proposition}\label{prop:existence-kink}
Let $f:\R\to\R$ be a $C^1$ function with $f(0)=0$ and define $F(s):=\int_0^sf(t)dt$. 
For $\omega\in\R$, let
\[
\zeta(\omega):=\inf\{ \zeta>0, F(\zeta)-\frac\omega2\zeta^2=0\}
\]
and assume that
there exists $\omega_1 \in\R$ such that 
\begin{equation}\label{eq:omega1}
\zeta(\omega_1)>0,\qquad f'(0)-\omega_1<0, \qquad f(\zeta(\omega_1))-\omega_1\zeta(\omega_1)=0.
\end{equation}
 Then  for $\omega=\omega_1$ there exists a kink profile solution  $\phi\in\mathcal C^2(\R)$ of \eqref{eq:snls}, i.e. $\phi$ is unique (up to translation), positive  and satisfies $\phi>0$, $\phi'>0$ on $\R$ and the boundary conditions
\begin{equation}\label{eq:connection}
\lim_{x\to-\infty}\phi(x)=0,\qquad \lim_{x\to+\infty}{\phi(x)}=\zeta(\omega_1)>0.
\end{equation}
If in addition
    \[
    f'(\zeta(\omega_1))-\omega_1<0,
    \]
then  for any $0<\delta<\omega_1- \max\{f'(0),f'(\zeta(\omega_1))\}$ there exists $C>0$ such that
    \begin{equation}\label{eq:exponential-decay}
    |\phi'(x)|+|\phi(x)\mathbf 1_{\{x<0\}} |+|(\zeta_1(\omega_1)-\phi(x) )\mathbf 1_{\{x>0\}}|\leq Ce^{-\delta|x|}.
    \end{equation} 
\end{proposition}

\begin{remark}
By uniqueness we mean that when $\omega=\omega_1$ the only solutions connecting $0$ to $\zeta(\omega_1)$ (i.e. satisfying \eqref{eq:connection}) are of the form $\phi(\cdot+c)$ for some $c\in\R$. 
\end{remark}

\begin{remark}
Using the symmetry $x\to -x$ it is easy to see that Proposition \ref{prop:existence-kink} also implies the existence and uniqueness of a kink solution $\phi$ satisfying
\[
\lim_{x\to-\infty}\phi(x)=\zeta(\omega_1)>0,\qquad \lim_{x\to+\infty}{\phi(x)}=0.
\]
Reverting the $>$ into the assumptions of Proposition \ref{prop:existence-kink} we immediately obtain the existence of a kink profile connecting $0$ to $\zeta(\omega_1)<0$.
 \end{remark}

\begin{remark} 
 It is well known (see \cite{BeLi83-1}) that if instead of \eqref{eq:omega1} 
we assume
that there exists $\omega_0 \in\R$ such that 
\begin{equation*}
\zeta(\omega_0)>0,\qquad f(\zeta(\omega_0))-\omega\zeta(\omega_0)>0,
\end{equation*}
 then for $\omega=\omega_0$ there exists a soliton profile, i.e. a unique positive even solution  $\phi\in\mathcal C^2(\R)$ to \eqref{eq:snls} with boundary conditions
\begin{equation*}
\lim_{x\to\pm\infty}\phi(x)=0.
\end{equation*}
\end{remark}

The profile on which we want to build a solution to \eqref{v100a0} is the following. Take $N\in\mathbb N$, $(v_j,x_j,\omega_j,\gamma_j)_{j=0,\dots,N+1}\subset\R^4$ such that $v_0<\dots<v_{N+1}$. Assume that for $\omega_0$ and $\omega_{N+1}$ there exist  two kink profiles $\phi_0$ and $\phi_{N+1}$ (solutions of \eqref{eq:snls}) satisfying the boundary conditions
\begin{align*}
& \lim_{x\to-\infty}\phi_0(x)\neq 0,&&	 \lim_{x\to+\infty}\phi_0(x)= 0,\\
& \lim_{x\to-\infty}\phi_{N+1}(x)= 0,&&	 \lim_{x\to+\infty}\phi_{N+1}(x)\neq 0.
\end{align*}
Denote by $K_0$ and $K_{N+1}$ the corresponding kinks.
For $j=1,\dots,N$, assume as before that we are given localized solitons profiles $(\phi_j)_{j=1,\dots,N}$ and let $R_j$ be the corresponding solitons. Consider the following approximate solution composed of a kink on the left and on the right and solitons in the middle (see Figure \ref{fig:multi-kink}):
\begin{equation}\label{eq:def-kinks-solitons}
KR(t,x):=K_0(t,x)+\sum_{j=1}^NR_j(t,x)+K_{N+1}(t,x).
\end{equation}

\begin{figure}[ht]
   \begin{center}
\includegraphics{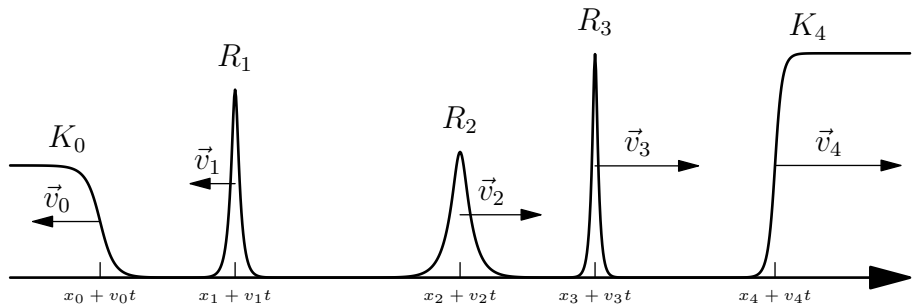}
   \end{center}
   \caption{Schematic representation of the multi-kink profile $KR$ in \eqref{eq:def-kinks-solitons}}
\label{fig:multi-kink}
\end{figure}

Our last result concerns solutions of that are composed of solitons and half-kinks.
\begin{theorem}\label{thm:multi-kinks}
Consider \eqref{v100a0} with $d=1$, $f(u)=g(|u|^2)u$
satisfying \eqref{v100a1}, and let $KR$ be the profile  defined in \eqref{eq:def-kinks-solitons}. Define $v_\star$ by 
\[
v_\star:=\inf\{|v_j-v_k|;\;j,k=0,\dots,N+1,\,j\neq k\}.
\]
Then there exist $v_\sharp>0$ (independent of $(v_j)$) large enough, $T_0\gg 1$
and constants $C,c_1,c_2>0$ such that if $v_\star>v_\sharp$, then
there exists a (unique) multi-kink solution
$u\in\mathcal{C}([T_0,+\infty),H^1_{\mathrm{loc}}(\mathbb R))$ to
  \eqref{v100a0} satisfying
\[
e^{c_1v_\star t}\norm{u-KR}_{S([t,+\infty))}+e^{c_2v_\star t}\norm{\nabla(u-KR)}_{S([t,+\infty))}
\leq C,\qquad\forall t\geq T_0. 
\]
\end{theorem}

It will be clear from the proof that the theorem remains valid if we
remove $K_0$ or $K_{N+1}$ from the profile $KR$. It is also fine if
$v_0>0$ or $v_{N+1}<0$.

\subsection{Strategy of the proofs}\label{subsec:strategy}

To simplify the presentation, we shall give a streamlined proof to Theorems \ref{thm_inf_1},  \ref{thm_N_1}, \ref{thm_N_2} and \ref{thm:multi-kinks}. The key tools
are Proposition \ref{prop_main} and Proposition \ref{prop_main_1} which reduce matters to the checking
of a few conditions on the solitons. This is done in Section \ref{sec:perturbation}. We stress that the situation
here is a bit different from the usual stability theory in critical
NLS problems (cf. \cite{KiVi08,LiZh11}). There the approximate solutions often have finite
space-time norms and the perturbation errors only need to be small
in some dual Strichartz space. In our case the solitary waves carry
infinite space-time norms on any non-compact time interval (unless
one considers $L_t^\infty$). For this we have to rework a bit the
stability theory around a solitary wave type solution. The price to
pay is that the perturbation errors and source terms need to be
exponentially small in time. This is the main place where the large
relative velocity assumption is used. We give the proofs of Theorems \ref{thm_N_1} and \ref{thm_N_2} in Section \ref{sec:N}, of Theorem \ref{thm_inf_1} in Section \ref{sec:infty} and finally of Theorem \ref{thm:multi-kinks} in Section \ref{sec:kinks}. In Section \ref{sec:T0}, we conclude the paper by giving three results similar to Theorem \ref{thm_N_2} with additional assumptions that allow us to take $T_0=0$.

\section{The perturbation argument}\label{sec:perturbation}

We start this section by giving some

\subsubsection*{Preliminaries and notations}
For any two quantities $A$ and $B$,  we use $A \lesssim B$ (resp. $A
\gtrsim B$ ) to denote the inequality $A \leq CB$ (resp. $A \geq
CB$) for a generic positive constant $C$. The dependence of $C$ on
other parameters or constants is usually clear from the context and
we will often suppress this dependence. Sometimes we will write $A
\lesssim_{k} B$ if the implied constant $C$ depends on the parameter
$k$. We shall use the notation $C=C(X)$ if the constant $C$
depends explicitly on some quantity $X$.

 For any function $f:\; \mathbb R^d\to \mathbb C$, we use $\|f\|_{L^p}$ or 
$\|f\|_p$ to denote the   Lebesgue $L^p$ norm of $f$ for
$1 \le p \le \infty$.
We use $L^q_t L^r_x$ to denote the space-time norm
$$ 
\| u \|_{L^q_t L^r_x(\R \times \R^d)} :=
 \Bigl(\int_\R \Bigl(\int_{\R^d} |u(t,x)|^r\ dx\Bigr)^{q/r}\ dt\Bigr)^{1/q},
$$
with the usual modifications when $q$ or $r$ are equal to infinity,
or when the domain $\R \times \R^d$ is replaced by a smaller region
of space-time such as $I \times \R^d$.  When $q=r$ we abbreviate
$L^q_t L^q_x$ as $L^q_{t,x}$ or $L^q_{tx}$. We shall write $u \in
L_{t,loc}^q L_x^r (\mathbb R \times \mathbb R^d)$ if
\begin{align}
\| u \|_{L_t^q L_x^r (K \times \mathbb R^d)} <\infty, \qquad\text{
for any compact $K \subset \mathbb R$}. \label{def_lqloc}
\end{align}

We shall need the standard dispersive inequality:  for any $2\le p \le \infty$,
\begin{align*}
\| e^{it\Delta} f \|_{p} \lesssim |t|^{-d(\frac 12 -\frac 1p)} \| f\|_{\frac p{p-1}}, \qquad\forall\, t\ne 0.
\end{align*}

The dispersive inequality can be used to deduce certain space-time estimates known
as Strichartz inequalities. Recall that for dimension $d\ge 1$, we say a pair of exponents $(q,r)$ is
\emph{(Schr\"odinger) admissible} if
\begin{align*}
\frac 2 q + \frac d r = \frac d2, \qquad 2\le q,r \le \infty, \; \text{and } (d,q,r)\ne (2,2,\infty).
\end{align*}
For any fixed space-time slab $I\times \mathbb R^d$, we define the Strichartz norm
\begin{align} \label{strichartz_def}
\| u \|_{S(I)} :=  \sup_{\text{$(q,r)$ admissible}} \|u \|_{L_t^q L_x^r (I\times \mathbb R^d)}.
\end{align}
For $d=2$, we need to further impose $q>q_1$ in the above norm for some $q_1$ slightly larger than $2$, so as to stay away from  the forbidden endpoint. The choice of $q_1$ is usually simple.
We use $S(I)$ to denote the closure of all test functions in $\R \times \R^d$ under this norm. We denote by
$N(I)$ the dual space of $S(I)$.

We now state the standard Strichartz estimates. For the non-endpoint case, one can
see for example \cite{GiVe92}. For the end-point case, see \cite{KeTa98}.

\begin{lemma} If $u: \; I \times \mathbb R^d \to
\mathbb C$ solves
\begin{align*}
i\partial_t u +\Delta u =F,\quad u(t_0)=u_0,
\end{align*}
for some $t_0 \in I$, $u_0 \in L_x^2(\mathbb R^d)$. Then
\begin{align}
\|u \|_{S(I)} \lesssim_d \|u_0\|_2 + \| F \|_{N(I)}. \notag
\end{align}
\end{lemma}

We need a few simple estimates on the nonlinearity. For any
complex-valued function $F=F(z)$, recall the notation
\begin{align*}
F_z:= \frac12 \left( \frac{\partial F}{\partial x} - i \frac{\partial F}{\partial y} \right),
\qquad F_{\bar z}:= \frac12 \left( \frac{\partial F}{\partial x} + i \frac{\partial F}{\partial y} \right).
\end{align*}
If we write $F(z)= F^*(z,\bar z)$ with $z$ and $\bar z$ treated as
independent variables in $F^*$, then $F_z = \frac{\partial
  F^*}{\partial z}$ and $F_{\bar z} = \frac{\partial F^*}{\partial
  \bar z}$.

By the chain rule and Fundamental Theorem of Calculus, it is easy to check that
\begin{align}
\nabla( F(u(x)) ) &= F_z(u(x)) \nabla u(x) + F_{\bar z}(u(x)) \overline{\nabla u(x)}; \notag \\
F(z_1)-F(z_2) &= (z_1-z_2) \int_0^1 F_z(z_2+\theta(z_1-z_2)) d\theta \notag\\
& \qquad+ (\overline{z_1-z_2}) \int_0^1 F_{\bar z} (z_2+ \theta(z_1-z_2)) d\theta.
 \label{v200a2}
\end{align}
These two identities will be used later.

\begin{lemma}[H\"older continuity of $f^{\prime}$ and $g$] \label{lem_6_2}
Let $f(z)=g(|z|^2)z$ for $z\in \mathbb C$ and suppose $g$ satisfy \eqref{v100a1} and \eqref{v100a3}.
 Then for all $ s_1,\,s_2>0$ we have 
\begin{multline}
|g(s_1^2)-g(s_2^2)|+|s_1^2g^{\prime}(s_1^2)-s_2^2 g^{\prime}(s_2^2)|\\
 \lesssim  |s_1-s_2|^{\min\{2\alpha_1,1\}}
(s_1+s_2)^{\max\{2\alpha_1-1,0\}}  \\
 \qquad\qquad+|s_1-s_2|^{\min\{2\alpha_2,1\}}
(s_1+s_2)^{\max\{2\alpha_2-1,0\}};\label{v200a4a}
\end{multline}
and for any $z_1$, $z_2\in \mathbb C$,
\begin{align}
 |f_z(z_1)-f_z(z_2)|+&|f_{\bar z}(z_1)-f_{\bar z}(z_2)| +|g(|z_1|^2)-g(|z_2|^2)| \notag \\
& \lesssim |z_1-z_2|^{\min\{2\alpha_1,1\}} (|z_1|+|z_2|)^{\max\{2\alpha_1-1,0\}} \notag\\
&\qquad\qquad +|z_1-z_2|^{\min\{2\alpha_2,1\}} (|z_1|+|z_2|)^{\max\{2\alpha_2-1,0\}}; \label{v200a4b} \\
|f(z_1)-f(z_2)| & \lesssim |z_1-z_2| \cdot \bigl( (|z_1|+|z_2|)^{2\alpha_1} + (|z_1|+|z_2|)^{2\alpha_2}
\bigr). \label{v200a4c}
\end{align}
\end{lemma}

\begin{proof}[Proof of Lemma \ref{lem_6_2}]
By \eqref{v100a1}, we get for any $s>0$,
\begin{align*}
|(s^2 g^{\prime}(s^2))^{\prime}| & \lesssim |sg^{\prime}(s^2)| + |s^3g^{\prime\prime}(s^2)|
 \lesssim s^{2\alpha_1-1}+s^{2\alpha_2-1}.
\end{align*}
Clearly for any $s_1, \, s_2> 0$, using the above estimate, we have
\begin{align*}
| s_1^2g^{\prime}(s_1^2)-s_2^2 g^{\prime}(s_2^2)|
& \lesssim  |s_1^{2\alpha_1} -s_2^{2\alpha_1}| + |s_1^{2\alpha_2} -s_2^{2\alpha_2} | \\
& \lesssim {\textstyle \sum_{k=1}^2} |s_1-s_2|^{\min\{2\alpha_k,1\}} (s_1+s_2)^{\max\{2\alpha_k-1,0\}}.
\end{align*}
The estimate for $g(s^2)$ is similar. Therefore \eqref{v200a4a} follows.
Observe that
\begin{align*}
f_z(z)=g^{\prime}(|z|^2)|z|^2 + g(|z|^2), \qquad f_{\bar z}(z)=g^{\prime}(|z|^2) z^2. 
\end{align*}
Obviously \eqref{v200a4b} holds for $g(|z|^2)$ and $f_z(z)$ using \eqref{v200a4a}.
For $f_{\bar z}(z)$, the estimate is similar: Let $z_1=\rho_1e^{i\theta_1}$, $z_2=\rho_2e^{i\theta_2}$, with $|\theta_1-\theta_2|\le \pi$.
One just need to note that
\begin{align*}
|f_{\bar z}(z_1)-f_{\bar z}(z_2)| = |g^{\prime}(\rho_1^2) \rho_1^2 e^{i(\theta_1-\theta_2)} - g^{\prime}(\rho_2^2)
\rho_2^2 e^{i(\theta_2-\theta_1)}|,
\end{align*}
and $|z_1-z_2| \sim |\rho_1-\rho_2| |\cos (\frac {\theta_1-\theta_2}2)| + (\rho_1+\rho_2)|\sin (\frac {\theta_1-\theta_2}2)|$.
Estimating the real and imaginary parts separately gives the result. 
Finally \eqref{v200a4c} follows from \eqref{v200a2} and \eqref{v200a4b}.
\end{proof}

With the preliminaries and notations out of the way, we now turn to the main matter of this section.

To prove our results,  we shall state and prove a general proposition on the solvability of
NLS around an approximate solution profile with exponentially decaying source terms. This proposition
is very useful in that it reduces the construction of multi-soliton solutions to the verification
of only a few conditions (see \eqref{asp_f_1a} and \eqref{asp_f_1c} below). To simplify numerology
we shall first deal with the pure power nonlinearity case.

\begin{proposition} \label{prop_main}
Let $0<\alpha<\alpha_{\max}$.
Let $\sourceterm=\sourceterm(t,x):\,[0,\infty)\times \R^d \to \mathbb C$,  $\profile=\profile(t,x):\,[0,\infty)\times \R^d \to \mathbb C$
be given functions  which satisfy for some
 $C_1>0$, $\lambda>0$:
\begin{align} \label{asp_f_1a}
\|\profile(t) \|_{\alpha+2}+    e^{\lambda t} \|\sourceterm(t)\|_{\frac{\alpha+2}{\alpha+1} }  \le C_1 , \qquad\forall\, t\ge 0.
\end{align}
Let $f_1(z)=|z|^\alpha z$ and consider
the equation
\begin{align}
\eta(t)= i \int_t^{\infty} e^{i(t-\tau)\Delta} \Bigl( f_1(\profile+\eta)
-f_1(\profile) + \sourceterm \Bigr)(\tau)\, d\tau. \label{17_g_100}
\end{align}

There exists a constant $\lambda_*=\lambda_*(\alpha,d,C_1)>0$ sufficiently large such that if
$\lambda\ge \lambda_*$ then the following holds:
\begin{itemize}
\item There exists a unique solution $\eta$ to \eqref{17_g_100} satisfying
\begin{align}
\|\eta(t)\|_{\alpha+2} \le C_1 e^{-\lambda t}, \qquad \forall\, t\ge 0. \label{17_g_102ddd}
\end{align}

\item All ($L^2$ level) Strichartz norms of $\eta$ are finite and  decay exponentially, i.e.
\begin{align}
\|\eta\|_{S([t,\infty))} \lesssim e^{-\lambda t}, \qquad \forall\, t\ge 0. \label{17_g_102}
\end{align}

\item If in addition to \eqref{asp_f_1a}, $(\sourceterm, \profile)$ also satisfies for some $C_2>0$:
\begin{align} \label{asp_f_1c}
\|\nabla  \profile(t) \|_{\alpha+2} + e^{\lambda t} \|\nabla  \sourceterm(t) \|_{\frac{\alpha+2}{\alpha+1}} \le C_2, \quad\forall\, t\ge 0,
\end{align}
then $\eta\in L^\infty_t H^1_x$, and for some $C_3=C_3(d,\alpha,C_1)>0$, 
\begin{align} \label{17_g_102bbb}
\|\nabla \eta(t) \|_{\alpha+2}+ \|\nabla \eta\|_{S([t,\infty))} \le C_3C_2 e^{- \min\{\alpha,1\} \lambda t}, \quad\forall \, t\ge 0.
\end{align}
Here both $C_3$ and $\lambda_*$ are independent of $C_2$.
\end{itemize}
\end{proposition}

\begin{proof}[Proof of Proposition \ref{prop_main}]
We write \eqref{17_g_100} as $\eta=V\eta$. We shall show that for $\lambda$ sufficiently large,
 $V$ is a contraction in the ball
\begin{align*}
B=\left\{ \eta:   \; \|\eta\|_{\tilde X}:=\Bigl\| e^{\lambda t} \| \eta(t) \|_{\alpha+2} \Bigr\|_{L_t^{\infty}([0,\infty))} \le C_1 \right\}.
\end{align*}

We first check that $V$ maps $B$ into $B$. Denote
\[
\theta:= d \left(\frac 12-\frac 1{\alpha+2}\right).
\]
It is easy to check that $0<\theta<1$ since by assumption $0<\alpha<\alpha_{\max}$. By the simple inequality
\begin{align}
|f_1(z_1)-f_1(z_2)| \lesssim |z_1-z_2| \cdot( |z_1|^\alpha +|z_2|^\alpha),\qquad \forall\, z_1,\,z_2\in \mathbb C
\label{00t0atmp13}
\end{align}
we have
\begin{align}
|f_1(\profile+\eta)-f_1(\profile)| \lesssim |\eta| \cdot (|\profile|^\alpha + |\eta|^\alpha). \label{00t0atmp9}
\end{align}

 By using the dispersive estimate,
 the assumptions on ($\profile$, $\sourceterm$) and \eqref{00t0atmp9},  we have
\begin{align}
&\| \eta(t) \|_{\alpha+2} \notag\\
& \le C \int_t^{\infty} |t-\tau|^{-\theta}
\Bigl( \| |\profile(\tau)|^\alpha |\eta(\tau)| \|_{\frac{\alpha+2}{\alpha+1}}+ \| |\eta(\tau)|^{\alpha+1} \|_{\frac{\alpha+2}{\alpha+1}}
+ \| \sourceterm(\tau) \|_{\frac{\alpha+2}{\alpha+1}} \Bigr) d\tau \notag \\
& \le C \int_t^{\infty} |t-\tau|^{-\theta}
\Bigl( \| \profile(\tau)\|^\alpha_{\alpha+2}  \|\eta(\tau) \|_{\alpha+2}+ \| \eta(\tau)\|_{\alpha+2}^{\alpha+1}
+ \| \sourceterm(\tau) \|_{\frac{\alpha+2}{\alpha+1}} \Bigr) d\tau \notag \\
& \le C \int_t^{\infty} |t-\tau|^{-\theta} \Bigl(
C_1^\alpha C_1 e^{-\lambda \tau} + C_1^{\alpha+1} e^{-\lambda(\alpha+1)\tau} + C_1 e^{-\lambda \tau} \Bigr) d\tau \notag \\
&\le C C_1 e^{-\lambda t} I_1,
\label{00t0atmp12}
\end{align}
where $C=C(d,\alpha)$ and ($\tilde \tau = \tau -t$)
\[
I_1 = C_1^\alpha \int_0^{\infty} (\tilde \tau)^{-\theta} e^{-\lambda
  \tilde \tau} d\tilde \tau + C_1^\alpha \int_0^{\infty} (\tilde
\tau)^{-\theta} e^{-\lambda(\alpha+1)\tilde \tau} d \tilde \tau +
\int_0^{\infty} (\tilde \tau)^{-\theta} e^{-\lambda \tilde \tau}
d\tilde \tau.
\]
It is not difficult to check that for $\lambda$ sufficiently large
\begin{align*}
CI_1 \le \frac{C(C_1,d,\alpha)}{\lambda^{1-\theta}} \le 1.
\end{align*}
 Hence $\|\eta(t)\|_{\alpha+2} \le C_1 e^{-\lambda t}$ and $V$ maps $B$ to $B$.
By using \eqref{00t0atmp13}
and a similar estimate as in \eqref{00t0atmp12}, we can also show that for any $\eta_1\in B$, $\eta_2\in B$,
\begin{align*}
\| (V\eta_1)(t) - (V\eta_2)(t) \|_{\tilde X} \le \frac 12  \| \eta_1 -\eta_2\|_{\tilde X}.
\end{align*}
This completes the proof that $V$ is a contraction on $B$.

Next \eqref{17_g_102} is a simple consequence of the Strichartz estimate. Denote by $a$ the number
such that $\frac 2a + \frac d {\alpha+2} = \frac d 2$. It is easy to check that $2<a<\infty$ since $0<\alpha<\alpha_{\max}$.
By \eqref{00t0atmp13} and Strichartz estimate, we have
\begin{align}
\| \eta \|_{S([t,\infty))}& \lesssim \| f_1(\profile+\eta) -f_1(\profile) \|_{L_{\tau}^{\frac a {a-1}} L_x^{\frac{\alpha+2}{\alpha+1}}
([t,\infty))}+ \| \sourceterm \|_{L_{\tau}^{\frac a {a-1}} L_x^{\frac{\alpha+2}{\alpha+1}}
([t,\infty))} \notag \\
& \lesssim \| |\eta| \cdot (|\profile|^{\alpha}+|\eta|^\alpha) \|_{L_{\tau}^{\frac a {a-1}} L_x^{\frac{\alpha+2}{\alpha+1}}
([t,\infty))}+ \| \sourceterm \|_{L_{\tau}^{\frac a {a-1}} L_x^{\frac{\alpha+2}{\alpha+1}}
([t,\infty))} \notag \\
& \lesssim \|\profile\|^\alpha_{L_{\tau}^\infty L_x^{\alpha+2}([0,\infty))} \cdot \| \eta \|_{L_\tau^{\frac{a}{a-1}}
L_x^{\alpha+2} ([t,\infty))} \notag\\
&\qquad+ \| \eta\|^{\alpha+1}_{L_{\tau}^{\frac{(\alpha+1)a}{a-1}} L_x^{\alpha+2}([t,\infty))}+
\| \sourceterm \|_{L_{\tau}^{\frac a {a-1}} L_x^{\frac{\alpha+2}{\alpha+1}}
([t,\infty))} \notag \\
& \lesssim e^{-\lambda t}, \qquad\forall\ t\ge 0. \label{tmp_100aaabbb}
\end{align}

Finally to show \eqref{17_g_102bbb}, we first prove that $V$ maps $B_1$ into $B_1$ where
\begin{align*}
B_1=B\bigcap \left\{ \eta:   \;  \sup_{t\ge 0} \Bigl( e^{  \min\{\alpha,1\}  \lambda t} \| \nabla \eta (t)
 \|_{\alpha+2} \Bigr) \le C_2 \right\}.
\end{align*}

We start with the identity
\begin{align}
&\nabla \bigl( f_1(\profile+\eta) - f_1(\profile) \bigr)  \notag \\
=&\;
\bigl( (\partial_z f_1)(\profile+\eta) -( \partial_z f_1)(\profile) \bigr) \nabla (\profile+\eta)
+(\partial_z f_1)(\profile) \nabla \eta \notag \\
&\; + \bigl( (\partial_{\bar z} f_1)(\profile+\eta) -( \partial_{\bar z} f_1)(\profile) \bigr) \overline{\nabla (\profile+\eta)}
+(\partial_{\bar z} f_1)(\profile) \overline{\nabla \eta}. \label{F1identity}
\end{align}

Note that for $0<\alpha\le 1$,
\begin{align*}
|(\partial_z f_1)(z_1) - (\partial_z f_1)(z_2) | \lesssim |z_1-z_2|^\alpha, \qquad \forall\, z_1, \, z_2 \in \mathbb C,
\end{align*}
and for $\alpha>1$,
\begin{align*}
| (\partial_z f_1) (z_1) - (\partial_z f_1)(z_2) | \lesssim
(|z_1|^{\alpha-1} +|z_2|^{\alpha-1} ) |z_1-z_2|, \qquad \forall\, z_1,\, z_2 \in \mathbb C.
\end{align*}

Therefore
\begin{align} \label{17_g_301}
|\nabla (f_1(\profile+\eta) -f_1(\profile) ) |
\lesssim
\begin{cases}
|\eta|^\alpha |\nabla (\profile+\eta) | +|\profile|^\alpha |\nabla \eta|, \quad \text{if $0<\alpha\le 1$}, \\
(|\eta|^{\alpha-1}+|\profile|^{\alpha-1}) |\eta| |\nabla(\profile+\eta) |+|\profile|^\alpha |\nabla \eta|,\quad
\text{if $\alpha>1$}
\end{cases}
\end{align}

 For simplicity we shall only discuss the case
$0<\alpha\le 1$. The argument for $\alpha>1$ is similar (even simpler) and will be omitted.
By using \eqref{17_g_301} ,
\eqref{17_g_102ddd}, \eqref{asp_f_1c},
and the dispersive
inequality, we have for $t\ge 0$:
\begin{align*}
\|\nabla \eta(t) \|_{\alpha+2}
& \lesssim_{d,\alpha}  \int_t^{\infty} |t-\tau|^{-\theta} \Bigl(
\| |\eta|^\alpha |\nabla(\profile +\eta) | \|_{\frac{\alpha+2}{\alpha+1}} + \| |\profile|^\alpha \nabla \eta \|_{\frac{\alpha+2}{\alpha+1}}
+\| \nabla \sourceterm \|_{\frac{\alpha+2}{\alpha+1}} \Bigr) d\tau \notag \\
& \lesssim_{d,\alpha} \int_t^{\infty} |t-\tau|^{-\theta} \Bigl( \|\eta\|_{\alpha+2}^\alpha (\| \nabla \profile\|_{\alpha+2} +\| \nabla \eta \|_{\alpha+2} )\notag\\
&\hspace{6cm}+\|\profile\|_{\alpha+2}^\alpha \| \nabla \eta \|_{\alpha+2} + \| \nabla \sourceterm\|_{\frac{\alpha+2}{\alpha+1}}  \Bigr)d\tau \notag\\
& \lesssim_{d,\alpha, C_1}  C_2 \int_t^{\infty} |t-\tau|^{-\theta} e^{-\lambda \alpha \tau} d\tau
 +C_2 \int_t^{\infty} |t-\tau|^{-\theta} e^{-\lambda \tau} d\tau \notag \\
& \lesssim_{d,\alpha,C_1} C_2 \int_t^{\infty}  |t-\tau|^{-\theta} e^{-\lambda \alpha \tau} d\tau \notag \\
& \le C_2 e^{-\lambda \alpha t} \cdot  C(d,\alpha,C_1)\int_0^{\infty} |\tilde \tau|^{-\theta} e^{-\lambda \alpha \tilde \tau} d\tilde \tau \notag \\
& = C_2 e^{-\lambda \alpha t} \cdot C(d,\alpha,C_1) \cdot (\lambda \alpha)^{-(1-\theta)} \int_0^{\infty}  |\tilde \tau|^{-\theta} e^{- \tilde \tau} d\tilde \tau.
\end{align*}
Now if we take $\lambda\ge \lambda_*$ and $\lambda_*=
\lambda_*(d,\alpha,C_1)$ is independent of $C_2$ and sufficiently
large such that
\begin{equation}
\label{prop23-H1small}
C(d,\alpha,C_1) \cdot (\lambda_* \alpha)^{-(1-\theta)} \int_0^{\infty}
|\tilde \tau|^{-\theta} e^{- \tilde \tau} d\tilde \tau \le \frac 12,
\end{equation}
then clearly
\begin{align*}
\| \nabla \eta(t) \|_{\alpha+2} \le C_2 e^{-\lambda \alpha t}, \qquad \forall\, t \ge 0.
\end{align*}
By a similar argument, we also obtain for the case $\alpha>1$,
\begin{align*}
\| \nabla \eta(t) \|_{\alpha+2} \le C_2 e^{-\lambda t }, \qquad \forall\, t \ge 0.
\end{align*}
Hence we have proved that $V$  maps $B_1$ to $B_1$.  Since $V$ is a contraction on $B$ and maps $B_1$ into $B_1$, it is obvious that  we have
constructed the solution satisfying
\begin{align} \label{18_rl_10a}
\| \nabla \eta (t) \|_{\alpha+2} \le C_2 e^{-\lambda \min\{\alpha,1\} t}, \qquad \forall\, t\ge 0.
\end{align}

It remains for us to bound the Strichartz norm $\| \nabla \eta(t) \|_{S([t,\infty))}$.
The argument is  similar to that in \eqref{tmp_100aaabbb}. Let $a$ be the same number such that
$\frac 2 a+ \frac d{\alpha+2} = \frac d 2$. By \eqref{17_g_301} and Strichartz, we have
\begin{align}
\| \nabla \eta \|_{S([t,\infty))} & \lesssim_d  \Bigl\| |\eta|^\alpha |\nabla (\profile+\eta) | \Bigr\|_{N([t,\infty))}
+ \Bigl\| |\profile|^\alpha |\nabla \eta | \Bigr\|_{N([t,\infty))} + \| \nabla \sourceterm \|_{N([t,\infty))} \notag \\
& \lesssim_d  \Bigl\| |\eta|^\alpha |\nabla \profile | \Bigr\|_{L_{\tau}^{\frac a{a-1}} L_x^{\frac{\alpha+2}{\alpha+1}}([t,\infty))}
+ \Bigl\| |\eta|^\alpha |\nabla \eta| \Bigr\|_{L_{\tau}^{\frac{a}{a-1}}  L_x^{\frac{\alpha+2}{\alpha+1}} ([t,\infty))}  \notag \\
& \qquad + \| |\profile|^\alpha |\nabla \eta |  \|_{L_{\tau}^{\frac a{a-1}} L_x^{\frac{\alpha+2}{\alpha+1}} ([t,\infty))} +
\|\nabla \sourceterm \|_{L_{\tau}^{\frac a{a-1}} L_x^{\frac{\alpha+2}{\alpha+1}}([t,\infty))} \notag \\
& \lesssim_d  \Bigl\|  |\eta|^\alpha  \Bigr\|_{L_{\tau}^{\frac a {a-1}} L_x^{\frac{\alpha+2}\alpha}([t,\infty))} 
 \Bigl\| |\nabla \profile| + | \nabla \eta | \Bigr\|_{L_{\tau}^{\infty} L_x^{\alpha+2} ([t,\infty))} \notag \\
&\hspace{2cm} + \| |\profile|^\alpha \|_{L_{\tau}^{\infty} L_x^{\frac {\alpha+2}\alpha}([t,\infty))} \|  \nabla \eta \|_{L_{\tau}^{\frac a{a-1}} L_x^{\alpha+2}([t,\infty))} \notag \\
&\hspace{4cm}+ \|\nabla \sourceterm \|_{L_{\tau}^{\frac a{a-1}} L_x^{\frac{\alpha+2}{\alpha+1}}([t,\infty))}. \label{21_10a}
\end{align}
By \eqref{17_g_102ddd},  we have
\begin{align*}
\Bigl\| |\eta|^\alpha \Bigr\|_{L_{\tau}^{\frac a{a-1}} L_x^{\frac {\alpha+2}\alpha} ([t,\infty))} &
\le \Bigl\| \|\eta\|_{\alpha+2}^\alpha \Bigr\|_{L_{\tau}^{\frac{a}{a-1}} ([t,\infty))} \notag \\
& \le C_1^\alpha \Bigl( \int_t^{\infty} e^{-\lambda \alpha \frac {a\tau} {a-1}} d\tau \Bigr)^{\frac{a-1}a} \notag \\
& \le C_1^\alpha \cdot \left( \lambda \alpha \frac a {a-1} \right)^{-\frac {a-1}a} \cdot e^{-\lambda \alpha t}.
\end{align*}
Plugging the above estimates into \eqref{21_10a} and using \eqref{asp_f_1c}, \eqref{18_rl_10a}, we obtain
\begin{align*}
\| \nabla \eta \|_{S([t,\infty))} \lesssim_{d,\alpha,C_1}C_2 e^{-\lambda \alpha t}.
\end{align*}
This settles the estimate for $0<\alpha\le 1$.

By a similar estimate, we also have for $\alpha>1$,
\begin{align*}
\| \nabla \eta \|_{S([t,\infty))} \lesssim_{d,\alpha,C_1}C_2 e^{-\lambda t}.
\end{align*}

This completes the proof of \eqref{17_g_102bbb}.
\end{proof}

The next proposition, unlike Proposition \ref{prop_main}, is based
solely on Strichartz estimates. It will be used in the proof of
Theorems \ref{thm_N_2} and \ref{thm:multi-kinks}.  Several assumptions
and conditions have to be modified to take care of the general
nonlinearity $f(u)$.

\begin{proposition} \label{prop_main_1}
Let $f$ be the same as in \eqref{v100a0} satisfying condition \eqref{v100a1}.
Let $\sourceterm=\sourceterm(t,x):\,[0,\infty)\times \R^d \to \mathbb C$,
$\profile=\profile(t,x):\,[0,\infty)\times \R^d \to \mathbb C$
be given functions  which satisfy for some
 $C_1>0$, $C_2>0$, $\lambda>0$, $T_0\geq 0$:
\begin{align} \label{v600a1}
&\|\profile(t) \|_{\infty}+    e^{\lambda t} \|\sourceterm(t) \|_2
 \le C_1, \qquad \forall\, t\ge T_0;\notag\\
&\|\nabla \profile(t) \|_2+ \| \nabla \profile(t) \|_{\infty} + e^{\lambda  t} \| \nabla \sourceterm(t)\|_2 \le C_2,
\qquad\;\forall\, t\ge T_0.
\end{align}
Consider the equation
\begin{align}
\eta(t)= i \int_t^{\infty} e^{i(t-\tau)\Delta} \Bigl( f(\profile+\eta)
-f(\profile) + \sourceterm \Bigr)(\tau)\, d\tau, \qquad t\ge T_0. \label{v600a2}
\end{align}
There exists a constant $\lambda_*=\lambda_*(d,\alpha_1, \alpha_2,
C_1)>0$ (independent of $C_2$) and a time $T_*=T_*(d,\alpha_1,\alpha_2,C_1,C_2)>0$  sufficiently large such that if $\lambda\ge \lambda_*$ and $T_0\ge T^*$,
then there exists a unique solution $\eta$ to
\eqref{v600a2} on $[T_0,+\infty)\times \mathbb R^d$ satisfying
\begin{align}
e^{\lambda t} \| \eta\|_{S([t,\infty))}+ e^{\lambda c_1 t} \| \nabla \eta \|_{S([t,\infty))}  \le 1, \qquad \forall t\ge T_0.\label{v600a3}
\end{align}
Here $c_1>0$ is a constant depending only on $(\alpha_1,d)$.
\end{proposition}

\begin{remark}
It is important to notice that $\lambda_*$ \emph{does not depend} on $C_2$. This will be essential for the proof of Theorems \ref{thm_N_2} and \ref{thm:multi-kinks}.
\end{remark}

\begin{proof}[Proof of Proposition \ref{prop_main_1}]
To minimize numerology we will suppress all explicit dependence of constants
on all parameters except the constant $C_2$.

We now sketch the main computations. Take $0<\beta_1 \le 2\alpha_1$ such that $\beta_1<\frac 1{100d}$. Denote
\begin{align*}
&\beta_2:=\begin{cases}
\frac 4{d-2}, \quad \text{if $d\ge 3$}, \\
m-1,  \quad \text{if $d=1,2$};
\end{cases}\\
 &c_1 :=\frac 12\beta_1.
\end{align*}
Here for $d=1,2$, $m$ is an integer such that $m>2\alpha_2+2$.

We shall omit the standard contraction argument since it will be essentially a repetition and we check only the following property: If on $[T_0,+\infty)$ we have
\[
e^{\lambda t} \| \eta\|_{S([t,\infty))}+ e^{c_1 \lambda t} \| \nabla \eta \|_{S([t,\infty))}  \le C.
\]
then the following a priori estimate holds, provided $\lambda$ and $T_0$ are chosen
large enough,
\begin{align}
e^{\lambda t} \| \eta\|_{S([t,\infty))}+ e^{c_1 \lambda t} \| \nabla \eta \|_{S([t,\infty))}  \le 1.
\label{vyy_apriori}
\end{align}

 We start with $\|\eta\|_{\stt}$.
By Lemma \ref{lem_6_2} and Strichartz, we have
\begin{align}
\|\eta\|_{S([t,\infty))} & \lesssim \| f(\profile+\eta) -f(\profile) \|_{\ntt} + \| \sourceterm\|_{\ntt} \notag \\
& \lesssim \| \eta( |\profile|^{\beta_1} + |\profile|^{\beta_2} + |\eta|^{\beta_1} + |\eta|^{\beta_2} )\|_{\ntt}
\label{bv802} \\
&\qquad+\|\sourceterm \|_{L_{\tau}^1 L_x^2([t,\infty))}. \label{bv801}
\end{align}
For \eqref{bv801}, by using \eqref{v600a1}, we have
\begin{align}
\| \sourceterm \|_{L_{\tau}^1 L_x^2([t,\infty))} \lesssim
\int_t^{\infty} e^{-\lambda \tau} d\tau \le \frac 1 {100} e^{-\lambda t }, \notag
\end{align}
where the constant $\frac{1}{100}$ is obtained by taking $\lambda$ large enough.\\
For \eqref{bv802}, consider two cases. If $d\ge 3$, then by the boundedness of $\profile$, we have
\begin{align}
\Bigl|\eta(|\profile|^{\beta_1} + |\profile|^{\beta_2} + |\eta|^{\beta_1} + |\eta|^{\beta_2} )
\Bigr| \lesssim |\eta|+|\eta|^{1+\frac 4{d-2}}.
\label{bv803}
\end{align}
Hence for $d\ge 3$, using that both $(\frac {2d+4}d, \frac {2d+4}d)$
and $(q^*,q)$ are admissible with $1/q* = 1/q-1/d = \frac{d-2}{2d+4}$,
\begin{align*}
\eqref{bv802} & \lesssim \| \eta \|_{L_{\tau}^1 L_x^2([t,\infty))} +
\| \eta |\eta|^{\frac 4{d-2}} \|_{L_{\tau,x}^{\frac{2(d+2)}{d+4}} ( [t,\infty))} \notag \\
& \lesssim \int_t^{\infty} e^{-\lambda \tau} d\tau + \| \eta \|_{L_{\tau,x}^{\frac{2(d+2)}d}([t,\infty))}
\cdot \| \eta \|^{\frac 4{d-2}}_{L_{\tau, x}^{\frac{2(d+2)}{d-2}} ( [t,\infty))} \notag \\
& \lesssim \frac 1 {\lambda} e^{-\lambda t} + \| \eta \|_{S([t,\infty))} \cdot \| \nabla \eta \|^{\frac 4{d-2}}_{S([t,\infty))}
\notag \\
& \lesssim \frac 1 {\lambda } e^{-\lambda t} + e^{-\lambda t} \cdot e^{-\frac 4{d-2} c_1 \lambda t} \notag \\
& \le \frac 1{100} e^{-\lambda t},
\end{align*}
where we have used the fact that $\lambda$ and $t\geq T_0$ are sufficiently large.\\
For $d=1,2$, we replace \eqref{bv803} by
\begin{align}
|\eta(|\profile|^{\beta_1} + |\profile|^{\beta_2} + |\eta|^{\beta_1} + |\eta|^{\beta_2} )| \lesssim |\eta|+|\eta|^{m}.
\notag
\end{align}
Then
\begin{align*}
\|  |\eta|^{m} \|_{\ntt} & \lesssim \| |\eta|^m \|_{L_{\tau}^1 L_x^2([t,\infty))} \notag \\
& \lesssim \int_t^{\infty} \| \eta(\tau) \|_{2m}^m d\tau.
\end{align*}
By \eqref{v600a3} and interpolation (i.e. Gagliardo-Nirenberg), we have for $\theta = d(\frac 12-\frac 1{2m})$
\begin{align*}
\|\eta(\tau) \|_{2m} &\lesssim \| \eta(\tau) \|_2^{1-\theta} \| \nabla \eta (\tau)\|_2^{\theta} \notag \\
& \lesssim  e^{-\bigl( (1-\theta) \lambda + c_1 \lambda \theta \bigr) \tau}.
\end{align*}
It is easy to check that $m(1-\theta)\ge 1$. Therefore
\begin{align*}
\| |\eta|^{m} \|_{\ntt}  \lesssim \int_t^{\infty} e^{-\lambda \tau} d\tau \le \frac 1 {100} e^{-\lambda t}.
\end{align*}
Hence the estimate also holds for $d=1,2$. Consequently for all $d\ge 1$, and $t\ge T_0$,
\begin{align*}
\| \eta\|_{\stt} \le \frac 1{10}e^{-\lambda t}.
\end{align*}

Now we estimate $\| \nabla \eta \|_{\stt}$. By Strichartz and
\eqref{F1identity}
\begin{align}
\| \nabla \eta \|_{\stt} & \lesssim \| \nabla (f(\profile+\eta) -f(\profile)) \|_{\ntt} + \| \nabla \sourceterm \|_{\ntt}
\notag \\
& \lesssim \| |f_z(\profile+\eta) -f_z(\profile) | \cdot \nabla (\profile+\eta) \|_{\ntt} \notag \\
&\qquad+ \| |f_{\bar z} (\profile+\eta) -f_{\bar z} (\profile) | \cdot \overline{\nabla (\profile+\eta)} \|_{\ntt}
 \notag \\
&\qquad + \| |f_z(\profile)| \nabla \eta \|_{\ntt}
+ \| |f_{\bar z} (\profile)| \overline{\nabla \eta} \|_{\ntt} + \| \nabla \sourceterm \|_{\ntt} . \notag
\end{align}
By Lemma \ref{lem_6_2}, we get
\begin{align}
\| \nabla \eta \|_{\stt} 
& \lesssim 
\| |\eta|^{\beta_1} |\nabla \eta |  \|_{\ntt}
+\| |\eta|^{\beta_1}  |\nabla \profile|) \|_{\ntt}
\label{vyy100a} \\
& \qquad + \| |\eta|^{\min\{\beta_2,1\}}(|\profile| + |\eta|)^{\max\{\beta_2-1,0\}} \cdot
(|\nabla \profile| + |\nabla \eta |) \|_{\ntt} \label{vyy100b}\\
&\qquad + \| (|f_z(W)|+|f_{\bar z}(W)|)\nabla \eta \|_{L_{\tau}^1 L_x^2([t,\infty))}+
\| \nabla \sourceterm\|_{L_{\tau}^1 L_x^2([t,\infty))}. \label{vyy100c}
\end{align}
Consider \eqref{vyy100a}.
Let $a$ be the number such that $\frac 2 a+ \frac d {\beta_1+2}=\frac d2$ and let $a^\prime=\frac a{a-1}$.
Then
\begin{align}
\| |\eta|^{\beta_1}  |\nabla \eta| \|_{\ntt}
&\lesssim \| |\eta|^{\beta_1} \nabla \eta \|_{L_{\tau}^{a^\prime} L_x^{\frac{\beta_1+2}{\beta_1+1}}([t,\infty))} \notag \\
&\lesssim \| |\eta|^{\beta_1} \|_{L_{\tau}^{(\frac 1{a^{\prime}}-\frac 1a)^{-1}} L_x^{\frac{\beta_1+2}{\beta_1}}
([t,\infty))} \| \nabla \eta \|_{L_{\tau}^a L_x^{\beta_1+2}([t,\infty))} \notag \\
& \lesssim \left( \int_t^{\infty}
\| \eta(\tau) \|^{\beta_1\cdot \frac a{a-2}}_{\beta_1+2} d\tau \right)^{\frac{a-2}a} \cdot \| \nabla \eta \|_{\stt}.
\label{vyy102}
\end{align}
It is not difficult to check that $\beta_1 \cdot \frac a{a-2}<a$
(since $\beta_1 < 4/d$). By using the fact
$\| \eta\|_{L_{\tau}^a L_x^{\beta_1+2}([t,\infty))} \lesssim e^{-\lambda t}$ and H\"older inequality,  for $t\ge T_0$ we have
\begin{align}
\int_t^{\infty} \| \eta(\tau) \|_{\beta_1+2}^{\beta_1 \cdot \frac a {a-2}} d\tau &\lesssim
\sum_{k\ge t-1} \int_k^{k+1} \| \eta(\tau) \|_{\beta_1+2}^{\beta_1 \cdot \frac a {a-2}} d\tau \notag \\
&\lesssim \sum_{k\ge t-1} \Bigl( \int_k^{k+1} \| \eta(\tau) \|_{\beta_1+2}^a d\tau \Bigr)^{\frac 1a \cdot\frac{a\beta_1}{a-2}} \notag \\
& \lesssim \sum_{k\ge t-1} e^{-\lambda k \cdot \frac{a\beta_1}{a-2}}
\lesssim \frac 1 {\lambda} e^{-\lambda (t-1) \cdot \frac{a\beta_1}{a-2}}. \notag 
\end{align}
Plugging the above estimate into \eqref{vyy102}, we obtain
\begin{align*}
\| |\eta|^{\beta_1}  |\nabla \eta| \|_{\ntt} \lesssim \left(\frac 1 {\lambda} \right)^{\frac{a-2}a}
 e^{-\lambda \beta_1(t-1)} \cdot e^{-c_1 \lambda t}
\le \frac 1{100} e^{-c_1 \lambda t}, \qquad t\ge T_0,
\end{align*}
for $\lambda$  sufficiently large and $ T_0\geq1$.\\
Similarly we have for $t\ge T_0$, using $\beta_1 a' = \beta_1 a/(a-1) < a$,
\begin{align*}
\| |\eta|^{\beta_1}  |\nabla \profile| \|_{\ntt}
&\lesssim \| |\eta|^{\beta_1} \|_{L_{\tau}^{a^{\prime} } L_x^{\frac{\beta_1+2}{\beta_1}}
([t,\infty))} \| \nabla \profile\|_{L_{\tau}^\infty L_x^{\beta_1+2}([t,\infty))} \notag \\
& \lesssim e^{-\lambda \beta_1 (t-1)} C_2 \notag \\
& \lesssim e^{-c_1 \lambda t} e^{-\lambda c_1 (t-2)} C_2 \le \frac 1{100} e^{-c_1 \lambda t}.
\end{align*}
Hence
\begin{align*}
\eqref{vyy100a} \le \frac 1{50} e^{-c_1 \lambda t}.
\end{align*}
Next we deal with \eqref{vyy100b}. Consider first the case $d\ge 6$. In this case $\beta_2\le 1$. Therefore
\begin{align}
\eqref{vyy100b} & \lesssim \| |\eta|^{\frac 4{d-2}} ( |\nabla \profile| + |\nabla \eta|) \|_{\ntt} \notag \\
& \lesssim \| |\eta|^{\frac 4{d-2}} \nabla \eta \|_{L_{\tau,x}^{\frac {2(d+2)}{d+4}}([t,\infty))}
+ \| |\eta|^{\frac 4{d-2}} |\nabla \profile| \|_{L_{\tau}^2 L_x^{\frac{2d}{d+2}} ([t,\infty))} \notag \\
& \lesssim \| \nabla \eta \|_{\stt}^{1+\frac 4{d-2}} + \left\|
\|\eta(\tau) \|_{L_x^{2}}^{\frac 4{d-2}} \cdot \| \nabla \profile\|_{L_x^{(\frac{d+2}{2d}-\frac 2{d-2})^{-1}}}
\right\|_{L_{\tau}^2([t,\infty))} \notag \\
& \lesssim e^{-c_1 \lambda (1+\frac 4{d-2}) t} + C_2 \cdot \Bigl( \int_t^{\infty}
\|\eta(\tau)\|_{2}^{\frac 4{d-2} \cdot 2}d\tau \Bigr)^{\frac 12} \notag \\
& \lesssim e^{-c_1 \lambda (1+\frac 4{d-2}) t} + C_2 \cdot \Bigl( \int_t^{\infty}
e^{-\frac{8}{d-2} \lambda \tau} d\tau\Bigr)^{\frac 12} \notag \\
& \le \frac 1 {200} e^{-c_1 \lambda t} + C_2 \cdot e^{-c_1 \lambda T_0}
\cdot e^{-c_1 \lambda t}
 \le \frac 1{100} e^{-c_1 \lambda t}, \label{vytmp_24aa}
\end{align}
for $\lambda$ and $T_0$ sufficiently large. \\
Consider next the case $3\le d\le 5$.  In this case $\beta_2=\frac
4{d-2}>1$. Therefore using the boundedness of $\profile$, we have
\begin{align}
\eqref{vyy100b} & \lesssim \| |\eta| \cdot (|\profile|+\eta)^{\frac 4{d-2}-1} (|\nabla \profile|+|\nabla \eta|) \|_{\ntt}
\notag \\
& \lesssim \||\eta|^{\beta_1} ( |\nabla \profile| + |\nabla \eta |) \|_{\ntt} + \| |\eta|^{\frac 4{d-2}}
(|\nabla \profile| + |\nabla \eta|) \|_{\ntt} \notag\\
& \lesssim |\eqref{vyy100a}| + \| |\eta|^{\frac 4{d-2}} \nabla \eta \|_{L_{\tau,x}^{\frac{2(d+2)}{d+4}}([t,\infty))}
+ \| |\eta|^{\frac 4{d-2}} |\nabla \profile| \|_{\ntt} \notag \\
& \le \frac 1{30} e^{-c_1 \lambda t}+ \| |\eta|^{\frac 4{d-2}} |\nabla \profile| \|_{\ntt}. \notag
\end{align}
For $d=5$, we can bound the term $\| |\eta|^{\frac 4{d-2}} |\nabla
\profile| \|_{\ntt}$ in the same way as in \eqref{vytmp_24aa} (it is
easy to check that $\frac{2d}{d+2}<\frac{d-2}2$ for $d\ge 5$).  For
$d=3,4$, we have
\begin{align}
\| |\eta|^{\frac 4{d-2}} |\nabla \profile| \|_{\ntt}
&\lesssim \| | \eta|^{\frac 4{d-2}} |\nabla \profile| \|_{L_{\tau}^2 L_x^{\frac{2d}{d+2}}([t,\infty))} \notag \\
& \lesssim C_2 \Bigl(\int_t^{\infty} \|\eta(\tau)\|_{\frac{8d}{d^2-4}}^{\frac 8{d-2}} d\tau\Bigr)^{\frac 12}.
\label{vytmp30}
\end{align}
Since $d=3, 4$, it is easy to check that $2<\frac{8d}{d^2-4}<\frac{2d}{d-2}$. By interpolation we have
for $\theta=\frac{1}{8} (d-2)^2$,
\begin{align*}
\|\eta (\tau) \|_{L_x^{\frac{8d}{d^2-4}}} & \lesssim
\| \eta(\tau)\|_2^{\theta} \| \nabla \eta (\tau) \|_2^{1-\theta} \notag \\
& \lesssim e^{-\theta \lambda \tau} e^{-(1-\theta)c_1 \lambda\tau} 
 \lesssim e^{-\theta \lambda \tau}.
\end{align*}
Plugging this estimate into \eqref{vytmp30}, we obtain for $d=3,4$,
\begin{align*}
\| |\eta|^{\frac 4{d-2}} |\nabla \profile| \|_{\ntt}  \lesssim C_2 \Bigl( \int_t^{\infty} e^{- \lambda (d-2) \tau} d\tau
\Bigr)^{\frac 12}
\lesssim C_2 \cdot \lambda^{-\frac {d-2}{2}} e^{-\frac{d-2}2 \lambda t} \le \frac 1{100} e^{-c_1 \lambda t}
\end{align*}
which is clearly enough for us.
\\
It remains to bound \eqref{vyy100b} for $d=1,2$. Since in this case $\beta_2=m-1>1$, we have
\begin{align}
\eqref{vyy100b} & \lesssim \| |\eta| (|\profile|+|\eta|)^{m-2} (|\nabla \profile|+|\nabla \eta|) \|_{\ntt} \notag \\
& \lesssim \| |\eta|^{\beta_1} (|\nabla \profile|+ |\nabla \eta|) \|_{\ntt} +
\| |\eta|^m |\nabla \profile| \|_{\ntt} + \| |\eta|^m |\nabla \eta| \|_{\ntt} \notag \\
& \lesssim |\eqref{vyy100a}| + \| |\eta|^m |\nabla \profile| \|_{L_{\tau}^1 L_x^2([t,\infty))}
+ \| |\eta|^m |\nabla \eta | \|_{L_{\tau,x}^{\frac{2(d+2)}{d+4}}([t,\infty))} \notag \\
& \lesssim |\eqref{vyy100a}| + C_2 \| \eta\|^m_{L_{\tau}^m L_x^{2m}([t,\infty))}
+ \|\nabla \eta\|_{\stt} \cdot \| \eta \|^m_{L_{\tau,x}^{m\cdot \frac{d+2}2}([t,\infty))}.
\label{vytmp65}
\end{align}
Now by Gagliardo-Nirenberg inequality,
\begin{align*}
\| \eta(\tau) \|_{2m}^m & \lesssim \Bigl( \| \eta(\tau)\|_2^{1-d(\frac 12-\frac 1{2m})}
\| \nabla \eta(\tau)\|_2^{d(\frac 12-\frac 1{2m})} \Bigr)^{m} \notag \\
& \lesssim \|\eta(\tau)\|_2^{\frac d2} \lesssim e^{-\frac 12 \lambda \tau}.
\end{align*}
Similarly
\begin{align*}
\| \eta(\tau) \|_{\frac{m(d+2)}2}^m \lesssim \| \eta(\tau)\|_2^{\frac{2d}{d+2}} \lesssim e^{-\frac 12 \lambda \tau}.
\end{align*}
Plugging the above estimates into \eqref{vytmp65} and integrating in time, we obtain for $d=1,2$,
\begin{align*}
\eqref{vyy100b} \le \frac 1{100} e^{-c_1 \lambda t}
\end{align*}
which is acceptable for us. We have completed the estimate of \eqref{vyy100b} for all $d\ge 1$.\\
Finally consider \eqref{vyy100c}.
Note $\||f_z(W)|+|f_{\bar z}(W)|\|_{L^\infty_{t,x}} \le C$  by \eqref{v200a4b} and
\eqref{v600a1}.
Thus
\begin{align*}
\eqref{vyy100c}&\le C \int_t^\infty \left ( \norm{\nabla
  \eta}_{L^\infty_t L^2_x([\tau,\infty))} +  \|\nabla H(\tau)\|_{L^2_x} \right )
  d\tau 
\\
&\le C \int_t^\infty (e^{- c_1 \lambda \tau} +
 C_2 e^{-\lambda \tau} )\, d\tau  \le \left(\frac C{c_1 \lambda} +C_2e^{-c_1\lambda t} \right)e^{-c_1\lambda t}  
\le \frac 1{100}  e^{-c_1\lambda t}  
\end{align*}
if we take $\lambda$ and $t\geq T_0$ large enough.\\
We have finished the proof of the a priori estimate \eqref{vyy_apriori}. The proposition is proved.
\end{proof}

\begin{remark}
Our proof does not work for the energy-critical case
   because the overlap of multi-solitons  no longer decays
   exponentially, but is just power-like; our proof relies heavily on
   the exponential decay property.
\end{remark}

\section{The $N$-soliton case}\label{sec:N}

In this section we give the proofs of Theorem \ref{thm_N_1} and Theorem
\ref{thm_N_2}.

We first recall \eqref{v100a6}, the multi-soliton profile, and observe
that the difference $\eta=u-R$ satisfies the equation
\begin{align}
i\partial_t \eta +\Delta \eta & = -f(R+\eta) +\sum_{j=1}^N f(R_j) \notag \\
& = -\bigl( f(R+\eta) -f(R) \bigr)-\bigl( f(R) -\sum_{j=1}^N
f(R_j)\bigr).\label{17_30a}
\end{align}

The following lemma gives the estimates on $R$ and the source term
$f(R)-\sum_{j=1}^N f(R_j)$.

\begin{lemma} \label{17_lem_1}
There exist constants $\tilde C_1>0$  depending on
$\bigl(N, \alpha_1,\alpha_2,d, (\omega_j)_{j=1}^N, (x_j)_{j=1}^N \bigr)$,
 $\tilde c_1>0$ depending only
on $\alpha_1$,
$\tilde C_2>0$ depending on $\bigl(N, \alpha_1,\alpha_2,d,
(\omega_j)_{j=1}^N, (v_j)_{j=1}^N, (x_j)_{j=1}^N \bigr)$, such that the following
hold: For every $1\le r\le \infty$ and $t \ge 0$,
\begin{align}
&\| R(t) \|_{r} +\sum_{j=1}^N \|R_j(t)\|_r\le \tilde C_1,
\label{17_source_40a} \\
& \Bigl\| f(R(t)) - \sum_{j=1}^N f(R_j(t)) \Bigr\|_{r} \le \tilde C_1
e^{-\tilde c_1 \sqrt{\omega_\star }  v_\star  t},  \label{17_source_40b} \\
& \| \nabla R(t) \|_{r} \le \tilde C_2, \label{17_source_40c} \\
&\Bigl \| \nabla \bigl(  f(R(t)) -\sum_{j=1}^N f(R_j(t)) \bigr)
\Bigr \|_{r} \le \tilde C_2  e^{-\tilde c_1 \sqrt{\omega_\star }v_\star t}. \label{17_source_40d}
\end{align}
Here recall  $\omega_\star =\min\{\omega_j, \, 1\le j\le N\}$ and $v_\star =\min\{|v_k-v_j|:\; 1\le k\ne j \le N\}$.
\end{lemma}

\begin{proof}[Proof of Lemma \ref{17_lem_1}]
The estimates \eqref{17_source_40a} and \eqref{17_source_40c} follow
directly from \eqref{v100a3a} and \eqref{v100a5}.

To simplify the notations, denote
\begin{align*}
\Omega&:=\bigl(N, \alpha_1,\alpha_2,d, (\omega_j)_{j=1}^N, (x_j)_{j=1}^N \bigr). 
\end{align*}

To prove
\eqref{17_source_40b}, we start with the point-wise estimate.
By \eqref{17_source_40a} and Lemma \ref{lem_6_2},
\begin{align}
\Bigl| f(R(t,x)) - \sum_{j=1}^N f(R_j(t,x)) \Bigr| &=\Bigl|
\sum_{j=1}^N g(|R(t,x)|^2)R_j(t,x) - \sum_{j=1}^N g(|R_j(t,x)|^2)
R_j(t,x) \Bigr| \notag \\
& \le \sum_{j=1}^N |g(|R(t,x)|^2) -g(|R_j(t,x)|^2) |\cdot |R_j(t,x)| \notag\\
 & \lesssim_{\Omega}  \sum_{j=1}^N \bigl(|R(t,x)-R_j(t,x)|+|R(t,x)-R_j(t,x)|^{2 \alpha_1}
 \bigr)
\cdot |R_j(t,x)| \notag \\
& \lesssim_{\Omega} \sup_{k \ne j} \Bigl( |R_k(t,x)| \cdot |R_j(t,x)|
+(|R_k(t,x)|\cdot |R_j(t,x)|)^{ 2\alpha_1} \Bigr).
\label{17_teqz_10}
\end{align}

It suffices to treat the first term in the bracket of \eqref{17_teqz_10}. The second term is similarly estimated.

By \eqref{v100a5}, for any $\delta<1$,
\begin{align*}
|R_k(t,x)| \lesssim_{d,\delta} e^{-\delta \sqrt{\omega_k}|x-v_kt-x_k|},  \quad \forall\, k=1,\dots, N.
\end{align*}
Now fix some $\delta<1$ for the rest of the proof.

Clearly for any $k\ne j$,
\begin{align}
|R_k(t,x)| \cdot |R_j(t,x)| \lesssim_{d,\delta}  e^{ - \delta (\sqrt{\omega_k} | x-v_k t -x_k|+
\sqrt{\omega_j}|x-v_jt -x_j| )} .\label{000v000a1}
\end{align}

By the triangle inequality, it is clear that for all $j\ne k$, $x\in
\mathbb R^d$, $t\ge 0$:
\begin{align}
& \sqrt{\omega_k}|x-v_kt -x_k|+\sqrt{\omega_j}|x-v_jt -x_j| \notag \\
\ge & \min\{\sqrt{\omega_j},\sqrt{\omega_k}\}
\Bigl( |v_j-v_k| t -|x_k-x_j| \Bigr) \notag \\
\ge & \sqrt{\omega_\star }
\Bigl( v_\star t -|x_k-x_j| \Bigr). \label{000v000a2}
\end{align}

Plugging \eqref{000v000a2} into \eqref{000v000a1}, we obtain
for any $k\ne j$,
\begin{align}
|R_k(t,x)| \cdot |R_j(t,x)| \lesssim_{\Omega} e^{-\frac{\delta}2 \sqrt{\omega_\star } v_\star  t}
\cdot e^{ - \frac{\delta}2 (\sqrt{\omega_k} | x-v_k t -x_k|+
\sqrt{\omega_j}|x-v_jt -x_j| )} \label{000v000a3}
\end{align}

Now \eqref{17_source_40b} follows easily from \eqref{000v000a3} and \eqref{17_teqz_10}.

Finally to show \eqref{17_source_40d} we only need to recall \eqref{v200a2} and write
\begin{align*}
&\nabla( f(R) ) -\sum_{j=1}^N \nabla (f(R_j) ) \notag \\
=&\;  \sum_{j=1}^N ( f_z(R)-f_z(R_j) ) \nabla R_j + \sum_{j=1}^N ( f_{\bar z}(R)-f_{\bar z} (R_j) ) \overline{\nabla R_j}.
\end{align*}
Thanks to the above decomposition, the rest of the proof is  essentially a repetition of that of
 \eqref{17_source_40b}. The only difference is that
the constants will depend on the velocities $v_j$ due to the terms $\nabla R_j$.  We omit further details.
\end{proof}

Now we are ready to complete the

\begin{proof}[Proof of Theorem \ref{thm_N_1}]
By \eqref{17_30a}, we need to solve the integral equation
\eqref{17_g_100} for $\eta$ on $[0,\infty) \times \mathbb R^d$, with
  $\profile=R$ and $\sourceterm = f_1(R) -\sum_{j=1}^N f_1(R_j)$.  By
  Lemma \ref{17_lem_1}, conditions \eqref{asp_f_1a} and
  \eqref{asp_f_1c} are satisfied. Thus, by Proposition
  \ref{prop_main}, there exists $\eta \in C([0,\infty),H^1)$ with $\|
    \langle \nabla \rangle \eta\|_{S([t,\infty))}$ decaying
      exponentially in $t$. Since the soliton piece $R \in
      C([0,\infty), H^1)$, so is $u(t)$.
\end{proof}

\begin{proof}[Proof of Theorem \ref{thm_N_2}]
This is similar to the proof of Theorem \ref{thm_N_1}.  We need to
apply Proposition \ref{prop_main_1} with $\profile=R$ and $\sourceterm
= f(R) -\sum_{j=1}^N f(R_j)$.  By Lemma \ref{17_lem_1}, the condition
\eqref{v600a1} is satisfied. By Proposition \ref{prop_main_1}, there
exists $\eta \in C([T_0,\infty),H^1)$ with $\| \langle \nabla \rangle
  \eta\|_{S([t,\infty))}$ (in particular $\|\eta(t)\|_{H^1}$) decaying
    exponentially in $t$. 
\end{proof}

\section{An infinite soliton train}\label{sec:infty}

In this section we construct an infinite soliton train solution to
\eqref{v100a0}.

Thanks to Proposition \ref{prop_main},
the proof of Theorem \ref{thm_inf_1} is reduced to checking the regularity of the infinite soliton
$R_\infty$ and the tail estimates.

\begin{lemma}[Regularity of $ R_\infty$] \label{lem_train_001}
Let $R_\infty$ be given as in \eqref{train_001} and recall $f_1(z)=|z|^{\alpha} z$. Then
\begin{enumerate}
\item There is a constant $\tilde A_1>0$ depending only on $(A_\omega,d,\alpha)$, such that
\begin{align}
&\|R_\infty(t) \|_{\infty} + \|R_\infty(t) \|_{r_1} + \sum_{j=1}^\infty
(\|\tilde R_j(t)\|_{\infty}+\|\tilde R_j(t)\|_{r_1}) \le \tilde A_1, \qquad\forall\, t\ge 0, \label{19_001} \\
& \| f_1(R_\infty(t)) \|_{\frac{\alpha+2-\epsilon_1}{\alpha+1}} + \sum_{j=1}^{\infty}
 \| f_1(\tilde R_j(t)) \|_{\frac{\alpha+2-\epsilon_1}{\alpha+1}} \le
\tilde A_1, \qquad\forall\, t\ge 0.                 \label{19_002}
\end{align}
where $0<\epsilon_1<1$ is a small constant depending on $(r_1,\alpha)$.

\item There are constants $\tilde c_1>0$, $\tilde c_2>0$ depending only on $(\alpha,d)$, $C_1>0$, $C_2>0$
depending on $(\tilde A_1,d, \alpha)$,  such that
\begin{align}
&\| f_1(R_\infty(t)) - \sum_{j=1}^{\infty} f_1(\tilde R_j(t)) \|_\infty \le C_1 e^{-\tilde c_1 v_\star t}, \qquad \forall\, t\ge 0,  \label{19_003} \\
&\| f_1(R_\infty(t)) -\sum_{j=1}^{\infty} f_1(\tilde R_j(t)) \|_{\frac{\alpha+2}{\alpha+1}} \le C_2 e^{-\tilde c_2 v_\star  t}, \qquad
\forall\, t\ge 0. \label{19_004}
\end{align}
\end{enumerate}

\end{lemma}

\begin{proof}[Proof of Lemma \ref{lem_train_001}]
The inequalities \eqref{19_001}--\eqref{19_002} are simple consequences of \eqref{omega_50}.
The proof of the inequality \eqref{19_003} is similar to the proof of \eqref{17_source_40b} and
we sketch the modifications. By using \eqref{19_001} and \eqref{v100a5} (fix $\eta<1$),
we have
\begin{align*}
|f_1(R_\infty(t,x)) -\sum_{j=1}^{\infty} f_1(\tilde R_j(t,x)) |
& \lesssim \sum_{j=1}^{\infty} \Bigl ||R_\infty(t,x)|^{\alpha} - |\tilde R_j(t,x)|^{\alpha} \Bigr| \cdot |\tilde R_j(t,x)| \notag \\
& \lesssim \sum_{j=1}^{\infty} |R_\infty(t,x)-\tilde R_j(t,x)|^{\min\{\alpha,1\}} |\tilde R_j(t,x)| \notag \\
& \lesssim \sum_{j=1}^\infty \Bigl|
\sum_{k\ne j} \omega_{k}^{\frac 1 {\alpha}}   e^{ - \eta \sqrt{\omega_k} | x-v_k t |} \Bigr|^{\min\{1,\alpha\}}
\omega_j^{\frac 1{\alpha}} e^{-\eta \sqrt{\omega_j}|x-v_j t |} \notag \\
& \lesssim
\sum_{j=1}^{\infty} \omega_j^{\frac 1\alpha} \Bigl | \sum_{k\ne j} \omega_k^{\frac 1{\alpha}}
e^{-\eta (\sqrt{\omega_k}|x-v_kt|+ \sqrt{\omega_j} |x-v_jt| )} \Bigr|^{\min\{1,\alpha\}}.
\end{align*}
By \eqref{omega_star_100}, we have
\begin{align*}
\sqrt{\omega_k}|x-v_kt|+ \sqrt{\omega_j} |x-v_jt|  \ge v_\star  t, \qquad\forall\, t\ge 0.
\end{align*}
Hence \eqref{19_003} follows from the above estimate and \eqref{omega_50}.
Finally \eqref{19_004} follows from interpolating the estimates \eqref{19_002}--\eqref{19_003}.
\end{proof}

We now complete the

\begin{proof}[Proof of Theorem \ref{thm_inf_1}]

We first rewrite \eqref{g_main_001} as
\begin{align*}
\eta(t) = i \int_t^{\infty} e^{i(t-\tau)\Delta}
\Bigl( f_1(R_\infty+\eta) -f_1 (R_\infty) + f_1(R_\infty) - \sum_{j=1}^{\infty} f_1(\tilde R_j) \Bigr)d\tau.
\end{align*}

We then apply Proposition \ref{prop_main} with $\profile=R_\infty$ and $\sourceterm = f_1(R_\infty) -\sum_{j=1}^{\infty} f_1(\tilde R_j)$.
By Lemma \ref{lem_train_001}, it is easy to check that the condition \eqref{asp_f_1a}
is satisfied. The theorem follows easily.
\end{proof}

\section{Half-kinks}\label{sec:kinks}

We conclude this paper by giving the proofs of Theorem \ref{thm:multi-kinks} and Proposition \ref{prop:existence-kink}.

\begin{proof}[Proof of Theorem \ref{thm:multi-kinks}]
The proof is similar to that of Theorem \ref{thm_N_2}. The only
difference is that, due to the non-zero background, the profile $KR$
is not in $\mathcal{C}(\R,H^1)$ any more but only in
$\mathcal{C}(\R,H^1_{\mathrm{loc}})$.
\end{proof}

\begin{proof}[Proof of Proposition \ref{prop:existence-kink}]
Assume $\omega=\omega_1$ and define $\zeta_1:=\zeta(\omega_1)$.
Take any $\phi_0\in(0,\zeta_1)$ and let $\phi$ be the solution to \eqref{eq:snls} on the maximal interval of existence $I$ and with initial data 
\[
\phi(0)=\phi_0,\qquad \phi'(0)=\sqrt{{\omega_1}\phi_0^2-2F(\phi_0)}.
\]

We first prove that $\phi(x)\in(0,\zeta_1)$ for any $x\in I$. Indeed, assume on the contrary that there exists $x_0$ such that $\phi(x_0)=0$ or $\phi(x_0)=\zeta_1$. From our choice of initial data for $\phi$, it follows that, for any $x\in I$, $\phi$ satisfies the first integral identity
\begin{equation}\label{eq:first-integral}
-\frac12|\phi'(x)|^2=F(\phi(x))-\frac{\omega_1}{ 2}|\phi(x)|^2.
\end{equation}
In particular, \eqref{eq:first-integral} at $x=x_0$ implies
\[
\phi'(x_0)=0.
\]
However, by Cauchy-Lipschitz Theorem it follows that $\phi\equiv 0$ or $\phi\equiv \zeta_1$ on $I$, which enters in contradiction with $\phi_0\in(0,\zeta_1)$. Hence for all $x\in I$ we have $\phi(x)\in(0,\zeta_1)$ which implies in particular that $I=\R$.

Since $\phi_0\in(0,\zeta_1)$, we have $\phi'(0)>0$ and by continuity $\phi'(x)>0$ for $x$ close to $0$. We claim that in fact $\phi'(x)>0$ on $\R$. Indeed, assume by contradiction that there exists $x_0$ such that $\phi'(x_0)=0$. From the first integral \eqref{eq:first-integral}, this implies that 
\[
F(\phi(x_0))-\frac{\omega_1 }{2}|\phi(x_0)|^2=0.
\]
Therefore $\phi(x_0)=0$ or $\phi(x_0)=\zeta_1$, but  we have proved that to  be impossible. Hence $\phi'>0$ on $\R$.

We consider now the limits of $\phi$ at $\pm \infty.$ Define 
\[
l:=\lim_{x\to-\infty}\phi(x),\qquad L:=\lim_{x\to+\infty}\phi(x).
\]
Let us show that $l=0$ and $L=\zeta_1$. Indeed, by \eqref{eq:first-integral}, we have  $F(l)-\frac{\omega_1} {2}l^2=0$ (indeed otherwise it would implies $|\phi'|>\delta>0$ for $x$ large, a contradiction with the boundedness of $\phi$). Since $\phi\in(0,\zeta_1)$ and $\phi$ is increasing, this implies $l=0$ and $L=\zeta_1$.

Let us now show that $\phi$ is unique up to translations. Assume by contradiction that there exists $\tilde\phi\in\mathcal C^2(\R)$ solution to \eqref{eq:snls} satisfying the connection property  \eqref{eq:connection}. Since we claim uniqueness only up to translation, we can assume that $\phi(0)\in(0,\zeta_1)$.
In addition, since we have shown that $\phi$ varies continuously from $0$ to $\zeta_1$, we can also assume without loss of generality that $\phi(0)=\phi_0=\tilde\phi(0)$. 
The first integral identity for $\tilde\phi$ is for any $x\in\R$
\begin{equation*}
\frac12|\tilde\phi'(x)|^2-\frac{\omega_1}{ 2}|\tilde\phi(x)|^2+F(\tilde\phi(x))=\frac12|\tilde\phi'(0)|^2-\frac{\omega_1}{ 2}|\tilde\phi(0)|^2+F(\tilde\phi(0))
\end{equation*}
In particular, since $\lim_{x\to\pm\infty}\tilde\phi'(x)=0$, and $0$ and $\zeta_1$ are zeros of $\zeta\to F(\zeta)-\frac\omega2\zeta^2$, we have 
\[
\frac12|\tilde\phi'(0)|^2=\frac{\omega_1}{ 2}|\tilde\phi(0)|^2-F(\tilde\phi(0)).
\]
As previously, it is not hard to see that $\phi'$ has a constant sign, which must be positive due to the limits of $\phi$ at $\pm\infty$. Therefore $\tilde\phi'(0)=\phi'(0)$ and the uniqueness follows from Cauchy-Lipschitz Theorem. Differentiating the equation we see that $\phi'$ verifies
    \[
    -(\phi')''+(\omega_1-f'(\phi))\phi'=0.
    \]
    Since $\lim_{x\to-\infty}(\omega_1-f'(\phi))=\omega_1-f'(0)>0$ and $\lim_{x\to+\infty}(\omega_1-f'(\phi))=\omega_1-f'(\zeta(\omega_1))>0$,     \eqref{eq:exponential-decay} follows from classical ODE arguments.
  \end{proof}

\section{Multi-soliton up to time zero}
\label{sec:T0}

In this section we add extra conditions to Theorem \ref{thm_N_2} so
that the solution exists in $[0,\infty)$.

\begin{theorem}
\label{thm_N_2a} Consider \eqref{v100a0} with $f(u)=g(|u|^2)u$
satisfying \eqref{v100a1} and \eqref{v100a3}. Let $R$ be the same as
in \eqref{v100a6} and define $v_\star $ as in \eqref{v100a11}. Suppose
\begin{equation}
\label{eq6.1}
\bar v :=\max_{k=1,\ldots,N} |v_k| \le M v_\star ^M, \quad \text{for
  some}\quad M \ge 1.
\end{equation}
There exist constants $C>0$, $c_1>0$, $c_2>0$ and $v_{\sharp}=v_{\sharp}(M)\gg 1$,
such that if $v_\star >v_{\sharp}$, then there is a unique solution
$u\in C([0,\infty), H^1)$ to \eqref{v100a0} satisfying
\begin{align*}
e^{c_1 v_\star t} \| u-R \|_{S([t,\infty))} +e^{c_2 v_\star t} \|
  \nabla (u -R) \|_{S([t,\infty))} \le C , \qquad \forall\, t\ge 0.
\end{align*}
\end{theorem}

\begin{remark}
 The extra condition \eqref{eq6.1} is satisfied for example if
$v_j = \mu \tilde v_j$ for some fixed $\tilde v_j$ and $\mu$ is an
increasing parameter.
\end{remark}

\begin{proof}[Sketch of proof]
Following the proof of Lemma
\ref{17_lem_1}, the assumption \eqref{v600a1} of Proposition
\ref{prop_main_1} is satisfied with 
\[
T_0 = 1, \quad \lambda = c v_\star, \quad C_1 = C_0, \quad C_2  = C_0 \bar v,
\]
where $c=C(\alpha_1)\sqrt{\min_{j=1,\dots,N}\{ \omega_j\}}$ and
$C_0=C_0(d,N,\alpha_1,\alpha_2,(\omega_j)_{j=1}^N, (x_j)_{j=1}^N)$ are
independent of $ (v_j)_{j=1}^N$.  The smallness condition used in the
proof of Proposition \ref{prop_main_1} is of the form
\begin{equation}
  \label{remark25-eq}
 e^{-c \lambda_* t} (1+C_2) \le \varepsilon
\end{equation}
for some small $\varepsilon>0$ independent of $C_2$. It can be satisfied either
by fixing $\lambda_*\gg 1$ independent of $C_2$ and then requiring
$t\ge T_0$ with $T_0=T_0(C_2)$ large (as in the proof of Proposition \ref{prop_main}), or by fixing $T_0=1$, using the assumption $C_2 = C_0
\bar v\le C_0 M v_\star^M$, and requiring $v_\star$ sufficiently large.  In
the latter case we get a solution $\eta(t)$ for $1 \le t <\infty$.
Since the soliton piece $R \in C([0,\infty), H^1)$ and
  $\|\eta(t=1)\|_{H^1}$ can be chosen sufficiently small by enlarging
  $\lambda_*$, we can extend $\eta(t)$ up to time $t=0$ with $O(1)$
  estimates by local existence theory in $H^1$.  
\end{proof}

The following result is $L^2$-theory for $L^2$-subcritical and
critical nonlinearities.

\begin{theorem}
\label{thm_N_2b} Consider \eqref{v100a0} with $f(u)=g(|u|^2)u$
satisfying \eqref{v100a1} and \eqref{v100a3}. Further assume $\alpha_2
\le 2/d$.  Let $R$ be the same as in \eqref{v100a6} and define
$v_\star $ as in \eqref{v100a11}.  There exist constants $C>0$,
$c_1>0$,  and $v_{\sharp}\gg 1$, such that if $v_\star
>v_{\sharp}$, then there is a unique solution $u\in C([0,\infty),
  L^2)$ to \eqref{v100a0} satisfying
\begin{align*}
e^{c_1 v_\star t} \| u-R \|_{S([t,\infty))} \le C , \qquad \forall\, t\ge 0.
\end{align*}
\end{theorem}

\begin{proof}[Sketch of proof]
We will modify the first part of the proof of Proposition
\ref{prop_main_1} which bounds $\eta=u-R$ in $\stt$. In that
  part, estimates for $\nabla \eta$ is only used to bound the global
  nonlinear terms $\sum_{j=1,2}|\eta|^{2\alpha_j+1}$ in the
  dual Strichartz space $\ntt$.  Suppose $\alpha_2 \le 2/d$ and
\[
\norm{\eta}_{\stt} \le e^{-\lambda t}, \quad \forall t>0.
\]
For $m =2\alpha_j+1$, $r=m+1$,  and $a$ such that $\frac 2a+\frac dr = \frac d2$, we have
\[
\norm{|\eta|^{m}}_{\ntt}\le \norm{|\eta|^{m}}_{L^{a'}L^{r'}(t,\infty)}\le 
\norm{\eta}^{m}_{L^{a'm}L^{r'm}(t,\infty)},
\]
where $r'=r/(r-1)$ and $a'=a/(a-1)$.
Let $q$ and $b$ be such that
\[
q=r'm, \quad \frac 2b+\frac dq = \frac d2.
\]
We claim that  $\alpha_j \le 2/d$ is equivalent to
\begin{equation}
\label{eq6.2}
a'm \le b.
\end{equation}
Indeed, \eqref{eq6.2} amounts to
\[
\frac 2{a'} \ge \frac {2m}{b} = m\left(\frac d2-\frac dq\right) 
= m\frac d2-\frac d{r'} ,
\]
i.e.
\[
m\frac d2\le \frac d{r'} +\frac 2{a'} = d+2 - \left(\frac d{r} +\frac 2{a}\right)=\frac d2 +2,
\]
which is exactly $\alpha_j \le 2/d$.
Thus
\begin{align*}
\norm{|\eta|^{m}}_{\ntt} 
&\le \bke{\int_t^\infty \norm{\eta(s)}_{L^q}^{a'm}ds}^{1/a'}
= \bke{\sum_{k=0}^\infty \int_{t+k}^{t+k+1} \norm{\eta(s)}_{L^q}^{a'm}ds}^{1/a'}
\\
&\le \bke{\sum_{k=0}^\infty \bke{\int_{t+k}^{t+k+1} \norm{\eta(s)}_{L^q}^b ds}^
{\frac{a'm}b}}^{1/a'}
\\
&\le \bke{\sum_{k=0}^\infty \bke{e^{-b\lambda (t+k) }}^ {\frac{a'm}b}}^{1/a'}
= C e^{-m\lambda t} .  
\end{align*}
We have used \eqref{eq6.2} in the second inequality. The rest of the
proof is the same as the first part of the proof for Proposition
\ref{prop_main_1}.
\end{proof}

The following result is valid for both $L^2$-subcritical and
$L^2$-supercritical nonlinearities. Its proof extends that of
Proposition \ref{prop_main}.
\begin{theorem}
\label{thm_N_2c} Consider \eqref{v100a0} with $f(u)=g(|u|^2)u$
satisfying \eqref{v100a1} and \eqref{v100a3}. Let $\beta_j=2\alpha_j$,
$j=1,2$, with $0<\beta_1 \le \beta_2< \alpha_{\max}$.  Assume for $d \ge 3$
\begin{alignat}{2}
\label{th64a}
&\frac{\beta_2 }{1+\beta_2} \le \beta_1 \le \beta_2, &\text{if}\quad 0
< \beta_2 < \frac{\alpha_{\max}}2,
\\
\label{th64b}
%
&\frac{\beta_2 }{\alpha_{\max}+1-\beta_2} < \beta_1 \le \beta_2,
\qquad &\text{if}\quad \frac{\alpha_{\max}}2 \le \beta_2 < \alpha_{\max},
\end{alignat}
and for $d=1,2$ we assume \eqref{th64a} only. 
Then we can choose $r_1$
and $r_2$ such that
\begin{align}
\label{th64eq1}
0 \le r_1-2 &\le \beta_1 \le \beta_2 \le r_2-2 < \alpha_{\max},\\
\label{th64eq2}
r_1\beta_2 &\le r_1r_2-r_1-r_2 \le r_2\beta_1.  
\end{align}
Let $R$ be the same as in \eqref{v100a6} and define $v_\star $ as in
\eqref{v100a11}.  For any choice of $r_1,r_2$ satisfying
\eqref{th64eq1}--\eqref{th64eq2}, there exist constants $C>0$,
$c_1>0$, and $v_{\sharp}\gg 1$, such that if $v_\star >v_{\sharp}$,
then there is a unique solution $u=R+\eta$ to \eqref{v100a0} on
$[0,+\infty)$ satisfying
\begin{equation}
\label{eq6.5}
\norm{\eta(t)}_{L^{r_1} \cap L^{r_2}} \le C e^{-c_1 v_\star t},\quad \forall t\ge 0.
\end{equation}
Moreover,
\[
\norm{\eta}_{\stt} \le C e^{-c_1 v_\star t},\quad \forall t\ge 0.
\]
\end{theorem}

Note the first strict inequality in \eqref{th64b}, compared to  \eqref{th64a}.
See Figure~\ref{fig:thm64} for the $\beta_1$-$\beta_2$ region when
$d=3$.  Remark also that \eqref{th64a} and \eqref{th64b} are equivalent (when $d\geq 3$) to 
\begin{align}
& \beta_1 \le \beta_2\le \frac{\beta_1}{1-\beta_1}, &&\text{if}\quad 0
< \beta_1 < \frac{\alpha_{\max}}{\alpha_{\max}+2},
\\
& \beta_1 \le \beta_2<\frac{(\alpha_{\max}+1)\beta_1}{1+\beta_1},
 &&\text{if}\quad \frac{\alpha_{\max}}{\alpha_{\max}+2} \le \beta_1 < \alpha_{\max}.
%
\end{align}

\begin{figure}[ht]
   \begin{center}
\includegraphics[scale=0.7]{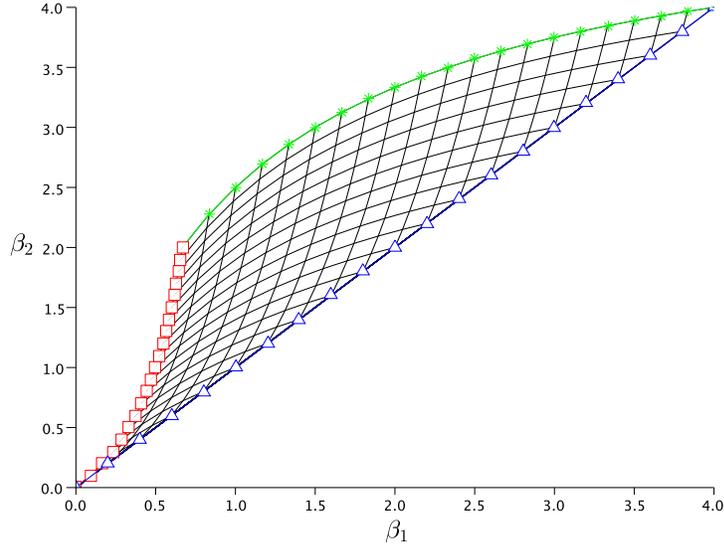}
   \end{center}
   \caption{Region of admissible $\beta_1,\beta_2$ in Theorem \ref{thm_N_2c} for $d=3$ }
\label{fig:thm64}
\end{figure}

\begin{proof}[Sketch of proof of Theorem \ref{thm_N_2c}]
For $j=1,2$ and $\theta_j = d(\frac 12 - \frac 1{r_j}) \in (0,1)$, we
have
\[
\norm{\eta(t)}_{L^{r_j}} \lesssim \int_t^\infty |t-\tau|^{-\theta_j}
\sum_{k=1,2} \norm{|\eta(\tau)|^{1+\beta_k}}_{r_j'} d\tau +
\text{(nice terms)},
\]
where $r_j'= r_j/(r_j-1)$.  The nice terms can be estimated as in the
proof of Proposition \ref{prop_main}.  Note that
\[
\norm{|\eta|^{1+\beta_k}}_{r_j'} = \norm{\eta}_{(r_j')(1+\beta_k)} ^{1+\beta_k}
\]
can be estimated by H\"older inequality and \eqref{eq6.5} if
\begin{equation}
\label{eq2}
r_1 \le \frac {r_j}{r_j-1} (1+\beta_k) \le r_2, \quad \forall j,k.
\end{equation}
For $j=1$, the left inequality of \eqref{eq2} is always true.  The
right inequality is equivalent to $ r_1(1+\beta_2) \le r_2(r_1-1)$, or
\begin{equation}
\label{eq3}
r_2 \le r_1(r_2-1-\beta_2) .
\end{equation}
For $j=2$, the right inequality of \eqref{eq2} is always true.  The
left inequality is equivalent to $ r_1(r_2-1) \le r_2(1+\beta_1)$, or
\begin{equation}
\label{eq4}
r_2(r_1-1-\beta_1) \le r_1.
\end{equation}
Equations \eqref{eq3} and \eqref{eq4} are equivalent to
\eqref{th64eq2}.  Furthermore, \eqref{th64eq1} and \eqref{th64eq2} can
be combined into the following equivalent condition
\begin{equation}
\label{th64eq3}
0 \le r_1-2 \le b_1(r_1,r_2) \le \beta_1 \le \beta_2 \le b_2(r_1,r_2)
\le r_2-2 < \alpha_{\rm max} 
\end{equation}
where
\[
b_1(r_1,r_2)= r_1-1-r_1/r_2,\qquad   b_2(r_1,r_2) = r_2-1-r_2/r_1.
\]
It turns out that when $2\le r_1 \le r_2<\alpha_{\rm max}+2$ we always have 
\[
0 \le r_1-2 \le b_1(r_1,r_2) \le b_2(r_1,r_2) \le r_2-2 < \alpha_{\rm max}.
\]
Thus for any $(\beta_1,\beta_2)$ in the right triangle with a vertex
$(b_1(r_1,r_2),b_2(r_1,r_2))$ and hypotenuse on the line
$\beta_1=\beta_2$, the pair $r_1,r_2$ satisfies \eqref{th64eq1}
and \eqref{th64eq2}.

Denote the curve $\Gamma(r_1)$ for fixed $2\le r_1 < 2+ \alpha_{\rm max}$,
\[
\Gamma(r_1)=\{ (b_1(r_1,r_2),b_2(r_1,r_2)): r_1\le r_2 \le 2+ \alpha_{\rm max}\}.
\]
It satisfies
\[
b_2 = \frac {b_1}{r_1-1-b_1}, \quad b_1 =
\frac{(r_1-1)b_2}{1+b_2},
\]
and starts at $(r_1-2,r_1-2)$. It goes to infinity with asymptote
$b_1=r_1-1$ for $d=1,2$, while ends at $\Sigma(2+ \alpha_{\rm
  max})$ to be defined below for $d\ge 3$.  It moves to the
right as $r_1$ increases. 

Denote the curve $\Sigma(r_2)$ for fixed $2<
r_2 \le 2+ \alpha_{\rm max}$,
\[
\Sigma(r_2)=\{ (b_1(r_1,r_2),b_2(r_1,r_2)): 2\le r_1 \le r_2\}.
\]
It satisfies
\[
b_2 = (r_2-1)\frac {b_1}{1+b_1}.
\]
It starts at $\Gamma(2)$ and ends at $(r_2-2,r_2-2)$.  It moves upward as
$r_2$ increases.

For given $0<\beta_1<\beta_2<\alpha_{\rm max}$, conditions
\eqref{th64a}--\eqref{th64b} imply that $(\beta_1,\beta_2)$ is on the
right of $\Gamma(2)$ and, if $d\ge 3$, is below $\Sigma(\alpha_{\rm
  max})$. Thus we can find $R_1=R_1(\beta_1,\beta_2)$ and
$R_2=R_2(\beta_1,\beta_2)$ such that $(\beta_1,\beta_2)$ is the
intersection point of $\Gamma(R_1)$ and $\Sigma(R_2)$, and $R_1 \le
R_2$. To satisfy \eqref{th64eq3}, we can either choose
$(r_1,r_2)=(R_1,R_2)$, or any $2\le r_1<R_1 \le R_2 < r_2<2+
\alpha_{\rm max}$ as long as the intersection point
$\Gamma(r_1)\cap\Sigma(r_2)$ is at upper-left direction to
$(\beta_1,\beta_2)$.

The above shows we can estimate $|\eta|^{1+\beta_k}$ in $L^{r_j'}$ for
$j,k=1,2$.

For the Strichartz estimate, since $(2/\theta_1,r_1)$ is admissible,
we have with $a=(2/\theta_1)'$
\begin{align*}
\norm{\eta}_{\stt} &\lesssim 
 \norm{f(W+\eta)-f(W)+H}_{L^a(t,\infty;L^{r_1'})} 
\\
&\lesssim \norm{e^{-c_1 v_* \tau}}_{L^a(t,\infty)}  \lesssim v_*^{-1/a}e^{-c_1 v_* t}.
\end{align*}
\end{proof}

\section*{Acknowledgments}
Part of this work was done when S. Le Coz was visiting the Mathematics
Department at the University of British Columbia, he would like to
thank his hosts for their warm hospitality. The research of S. Le Coz
is supported in part by the french ANR through project ESONSE.  The
research of Li and Tsai is supported in part by grants from 
the Natural Sciences and
Engineering Research Council of Canada.

\bibliographystyle{abbrv} 
\bibliography{biblio}

\end{document}